\documentclass[final,1p,times]{elsarticle2}

 \usepackage{amssymb}
 \usepackage{amsthm}

\usepackage{amsmath}
\usepackage{amsfonts}
\usepackage{latexsym}
\usepackage{mathrsfs}
\usepackage{times}
\usepackage{array}
\usepackage{amscd}
\usepackage{a4}
\usepackage{pstricks}
\input{boxedeps} %% windows
\SetEPSFDirectory{} %% windows
\HideDisplacementBoxes
\SetRokickiEPSFSpecial  %% dvips by Tom Rokicki
\newtheorem{theorem}{Theorem}
\newtheorem{proposition}{Proposition}
\newtheorem{definition}{Definition}[section]
\newtheorem{lemma}{Lemma}
\newdefinition{example}{Example}
\newdefinition{remark}{Remark}
\input xy 
\xyoption{all} 
\xyoption{2cell} 
\DeclareMathAlphabet{\ams}{U}{msb}{m}{n}\DeclareMathAlphabet{\goth}{U}{euf}{m}{n}

\def\sl{S\kern-1pt L}

\def\ml{M\kern-1pt L}
\def\sp{S\kern-1.5pt p}

\def\hom{\text{Hom}}
\def\id{\text{id}}
\def\ov{\overline}

\def\ve{\varepsilon}
\def\aa{\alpha}
\def\ss{\sigma}
\def\ve{\varepsilon}
\def\Z{\ams{Z}}
\def\R{\ams{R}}

\def\0{\mathbf{0}}
\def\1{\mathbf{1}}
\def\quo{/\kern -.45em\sim}

\def\Langle{\langle\kern -2pt\langle}
\def\Rangle{\rangle\kern -1.9pt\rangle}
\newpsobject{showgrid}{psgrid}{subgriddiv=1,griddots=10,gridlabels=6pt,gridcolor=red}
\newcommand{\prsh}{\mathbf{PreSh}}
\newcommand{\op}{\text{op}}
\newcommand{\invlim}{\varprojlim}
\DeclareMathOperator{\Ker}{ker} 
\DeclareMathOperator{\im}{im} 
\DeclareMathOperator{\coker}{coker} 

\newcommand{\ra}{\rightarrow}
\newcommand{\rk}{rk}
\newcommand{\cork}[1]{|\kern0.75pt{#1}\kern1pt|}

\newcommand{\ab}{\mathbf{Ab}}
\newcommand{\BP}{{\bf P}} 
\newcommand{\BPn}[1]{{\BP}^{{#1}}} 
 
\newcommand{\BPg}[1]{{\bf P}_{>{#1}}} 
\newcommand{\BPge}[1]{{\bf P}_{\geq {#1}}} 
\newcommand{\BQ}{{\bf Q}}
 
\newcommand{\BR}{{\bf R}}  
\newcommand{\hs}{H\kern-0.5pt S}
\newcommand{\htt}{H\kern-0.5pt T}
\newcommand{\hc}{H\kern-0.5pt C}

\newcommand{\geo}[1]{|{{#1}}|}

\journal{Journal of Algebra}

\begin{document}

\begin{frontmatter}

%% Title, authors and addresses

%% use the tnoteref command within \title for footnotes;
%% use the tnotetext command for the associated footnote;
%% use the fnref command within \author or \address for footnotes;
%% use the fntext command for the associated footnote;
%% use the corref command within \author for corresponding author footnotes;
%% use the cortext command for the associated footnote;
%% use the ead command for the email address,
%% and the form \ead[url] for the home page:
%%
%%\title{Title\tnoteref{label1}}
%% \tnotetext[label1]{}
%%\author{Name\corref{cor1}\fnref{label2}}
%%\ead{email address}
%% \ead[url]{home page}
%% \fntext[label2]{}
%% \cortext[cor1]{}
%% \address{Address\fnref{label3}}
%% \fntext[label3]{}

\title{Cellular cohomology of posets with local coefficients\tnoteref{label1}}
\tnotetext[label1]{Dedicated to the memory of Colin Maclachlan}

%% use optional labels to link authors explicitly to addresses:
%% \author[label1,label2]{<author name>}
%% \address[label1]{<address>}
%% \address[label2]{<address>}

\author[bje]{Brent Everitt}
\ead{brent.everitt@york.ac.uk}
\author[prt]{Paul Turner}
\ead{prt.maths@gmail.com}
\address[bje]{Department of Mathematics, University of York, York
YO10 5DD, United Kingdom.}
\address[prt]{Section de math\'ematiques,  
Universit\'e de Gen\`eve, 2-4 rue du Li\`evre, CH-1211, Geneva,
Switzerland.}

\begin{abstract}
We describe a ``cellular'' approach to the computation of the cohomology
of a poset with coefficients in a presheaf. A cellular cochain complex
is constructed, described explicitly and shown to compute the
cohomology under certain circumstances. The descriptions are refined
further for certain classes of posets including the cell posets of
regular CW-complexes and geometric lattices. 
\end{abstract}

%\begin{keyword}
%Coxeter groups \sep hyperplane arrangements \sep inverse semigroups
%\sep algebraic monoids \sep Renner monoids
%% keywords here, in the form: keyword \sep keyword

%% PACS codes here, in the form: \PACS code \sep code

%% MSC codes here, in the form: \MSC code \sep code
%% or \MSC[2008] code \sep code (2000 is the default)
%\end{keyword}

\end{frontmatter}

\section*{Introduction}

The cohomology groups of a poset $\BP$ with coefficients in a presheaf $F$ are
the derived functors (or higher limits) of the limit
$\invlim_{\BP} F$, and they can be computed using a canonical complex $S^*(\BP;F)$.
In \cite{EverittTurner3} we showed that 
the Khovanov homology groups of a link diagram are the higher limits of a
certain poset and presheaf.  The definition of Khovanov homology
 involves a different complex however, which, while based on essentially the same
poset and presheaf,  appears ad hoc in its
construction. The motivation for  \cite{EverittTurner3} was to place
it within a more general framework. The relationship between the two constructions is 
analogous to that found in topology, where the cellular
chain complex of a space is simple and  well suited to explicit
computation -- like the definition of Khovanov homology -- but has less flexibility 
than the singular chain complex. 

So we are naturally led to the following questions:
for what posets $\BP$ and presheaves $F$ can a ``cellular''
cohomology be defined? Under what circumstances will this  
compute the higher limits, that is to say, coincide with the usual cohomology?
In this paper we propose a definition of  cellular cohomology $\hc^*
(\BP; F)$  applicable to  a large class of posets, and show that for 
many naturally occurring examples, including the cell posets of
regular CW-complexes and geometric lattices, this cellular
cohomology computes the higher limits. 

Specifically, we define a cochain complex $C^* (\BP; F)$ for a
(graded) poset $\BP$  equipped with a presheaf $F$ 
by mimicking the construction of the cellular chain complex in
topology. We define first the
relative cohomology of pairs and then apply this to adjacent degrees of a filtration of $\BP$. 
The role of open cells is played by open intervals $\BPg x $.
Our main
result is that, as in topology, a vanishing
condition on these relative cohomologies suffices for this cellular complex
to compute the higher limits. If a poset has this condition we call it
\emph{cellular\/} and our main result (see Section \ref{section3})
then reads:

\begin{theorem}
Let $\BP$ be graded, cellular, locally finite with a corank function and let $F$ be a
presheaf on $\BP$. Then there is a natural isomorphism
$$
\hs^* (\BP; F)\cong \hc^* (\BP; F).
$$
\end{theorem}

The proof is in
two steps: finite posets are handled via a spectral sequence and this
is extended to locally finite posets using a projective resolution. On
the other hand it is not hard to find examples for which the higher
limits \emph{cannot\/} be computed 
cellularly.  

We also spend some time describing the cellular
cochain complex explicitly. For example in Section \ref{section2}
we describe the cochain groups as a product taken over the ``cells'': 

\begin{proposition}
Let $\BP$ be graded with a corank function and $F$ a presheaf on $\BP$.
Then there are natural isomorphisms
$$
C^n  (\BP; F)\cong \prod_{\cork{x} = n}\widetilde{\hs}^{n-1}(\BPg x ; \Delta F(x))  
$$
\end{proposition}

Our most concrete result concerning the complex itself is Proposition
\ref{prop:cellularcomplex.v4} where we show that the form of the cellular
cochain groups is shaped by abelian groups that reflect the
structure of closed intervals $\BPge{x}$. This also makes it easy to see that
the cochain groups need not necessarily be free, even if the presheaf takes values
that are free.

% In \S\ref{section4:1}  we show that posets $\BP$ with a  unique
% maximum (that are cellular) have higher limits that vanish in all
% non-zero degrees %$\invlim_\BP^i F=0$ for $i>0$ 
%  for any
% presheaf $F$. The proof is rendered almost trivial by the use of cellular cohomology. 
% If $\BP$ has instead a
% unique minimum (\S\ref{section4:3}) then by contrast 
% the higher limits can be almost
% arbitrarily complicated; the Khovanov homology result mentioned
% above is an example. Our point here is to dispel a
% commonly held misconception: that when we have 
% a unique extremum the cohomology vanishes in all but one degree, no 
% matter what the presheaf. 
Section \ref{section4} describes in some detail
the cellular complex in a couple of important cases: when $\BP$ is the cell poset of a regular
CW-complex (this includes the Khovanov homology result mentioned
above) and when $\BP$ is a geometric lattice. 

\paragraph{} \emph{Colin Maclachlan\/}, colleague, mentor and friend, died in
November 2012. Colin maintained a healthy skepticism of ``abstract mumbo
jumbo'', so we're not sure that he would have approved of
this paper. Nevertheless, we dedicate it to him with much respect and
affection.

\section{Cohomology of posets with
  coefficients in a presheaf}\label{section1}

We start by recalling definitions and elementary
results concerning the cohomology of posets with coefficients in a
presheaf. Everything is well-known (see
\cite[Appendix II.3]{Gabriel-Zisman67})
but it is convenient to have it to
hand and set down the notation we will be using. Many of the
definitions and results carry straight through to the more general setting of small
categories, but as we will be working with posets we prefer to state
everything in these more concrete terms. % The reader is referred to
% \cite{??Moerdijk}, \cite{GZ} for more on these topics.

\subsection{Definitions}
\label{section1:subsection1}

Let $\BP=(\BP,\leq)$ be a poset.
We will usually think of $\BP$ as a category 
having objects the elements of
the poset and with a unique morphism $x\ra y$ if and only if $x\leq
y$. The nerve $N^*\BP$ of $\BP$ is the simplicial
set with $n$-simplicies
$N^n\BP$ the poset sequences $\sigma =\sigma_n\leq
\cdots \leq \sigma_0$, where the $\sigma_i\in \BP$,
and with face maps $d_i:N^n\BP\rightarrow N^{n-1}\BP$ and
degeneracy maps $s_i:N^{n}\BP\rightarrow N^{n+1}\BP$
given by
\begin{equation}
  \label{eq:23}
d_i\sigma  = \sigma_n\leq\cdots \leq \widehat{\sigma}_i \leq \cdots \leq \sigma_0
\text{ and }
s_i\sigma =  \sigma_n\leq\cdots \leq
\sigma_i
\leq \sigma_i \leq \cdots \leq \sigma_0.  
\end{equation}
The geometrical realization of $N^*\BP$ will be denoted by $\geo {N^*\BP}$. It is a
CW-complex  with a single $n$-cell
for each non-degenerate $n$-simplex $\ss=\ss_n<\cdots<\ss_0$. 

A presheaf on $\BP$ is a (covariant) functor 
$F:\BP^\op\rightarrow\ab$ to abelian groups (or, more generally, to
$R$-modules; again, we specialize for concreteness). 
The category
$\prsh(\BP)$ has as objects the presheaves
$F:\BP^\op\rightarrow\ab$ and as morphisms the natural
transformations $\kappa:F\rightarrow G$. We write 
$F^y_x$ for the homomorphism $F(y)\ra F(x)$ induced by $x\leq y$ in $\BP$,
and $\kappa_x$ for the map $F(x)\rightarrow G(x)$
that is the component at $x$ of the natural transformation $\kappa$.
%For any $A\in\ab$ write $\Delta A$ for 
%$1:\Delta A(y)=A\ra A=\Delta A(x)$ for any $x\leq y$ in $\BP$.
% {\red It turns out that the presheaves thus defined are equivalent to the \emph{sheaves\/}
%   on $\BP$ when it is equipped with the order topology, where the open sets are the
%   $X\subseteq\BP$ such that $x\in X$ and $x\leq y$ implies $y\in
%   X$. One can then appeal directly to classical sheaf theory (as in
%   \cite{Godement73}) for much of the following. We will maintain a combinatorial
% view and keep the presheaf terminology.}

\begin{example}
For $A\in\ab$ the \emph{constant\/} presheaf $\Delta A$ is
defined by $\Delta A(x)=A$ for every $x\in\BP$ and 
$(\Delta A)^y_x=1$ for every $x\leq y$ in $\BP$. 
\end{example}

\begin{example}
For $A\in\ab$ and $x\in\BP$ the \emph{Yoneda\/}  presheaf
$\Upsilon_xA$ is defined by 
$$
\Upsilon_xA(y)=
\left\{
\begin{array}{ll}
A,&\text{if }y\leq x\\
0,&\text{otherwise},
\end{array}
\right.
$$
and with $(\Upsilon_xA)_y^z=1$ when $y\leq z\leq x$; or $0\ra A$ when
$y\leq x$ and $z\not\leq x$, and the zero
map otherwise. Thus
$\Upsilon_xA$ is the constant presheaf $\Delta A$ on the closed interval $\BP_{\leq
  x}=\{y\in\BP\,|\,y\leq x\}$ and the zero presheaf 
on the rest of $\BP$. If $x\leq y$ in $\BP$ and $A\ra B$ is a map of abelian
groups, then there is an induced morphism of presheaves
$\Upsilon_xA\ra\Upsilon_yB$. 
Indeed, the most useful property of the Yoneda functor $\Upsilon_x$ is
that is it left adjoint to the evaluation functor 
 $\prsh(\BP) \ra \ab$ taking $F$ to $F(x)$. Explicitly, this adjunction is 
$\hom_{\prsh(\BP)}(\Upsilon_xA,F)\cong\hom_\Z(A,F(x))$, via
$\kappa\mapsto\kappa_x$, and if $\zeta:\Upsilon_xA\ra\Upsilon_yB$ is the
induced morphism above, then $\hom_{\prsh(\BP)}(\relbar,F)$
  applied to $\Upsilon_xA\ra\Upsilon_yB$ is the map
\begin{equation}
  \label{eq:21}
% \hom_{\prsh(\BP)}(\Upsilon_xA\ra\Upsilon_yB,F)
% =\hom_\Z(B,F(y))\ra\hom_\Z(A,F(x))  
\hom_\Z(B,F(y))\stackrel{\zeta^*}{\ra}\hom_\Z(A,F(x))  
\text{ with }
\kappa_y\mapsto
F^y_x\kappa_y\zeta_x.
\end{equation}
In particular $\Upsilon_xA$
is projective in $\prsh(\BP)$ if and only if $A$ is projective in
$\ab$ (i.e.: $A$ is free). 
\end{example}

For any presheaf $F$ the limit $\invlim_\BP F$ (or group of
  global sections) exists in $\ab$ 
and is constructed by taking the subgroup of the product
$\prod_{x\in\BP}F(x)$ consisting of those $\BP$-tuples
$(a_x)_{x\in\BP}$ with $a_x\in F(x)$, and such that for all $x\leq y$ in $\BP$
the induced map $F(y)\rightarrow F(x)$ sends
$a_y$ to $a_x$. 
Indeed we have a left
exact functor $\invlim_\BP:\prsh(\BP)\rightarrow\ab$ and the right derived
functors 
$$
\textstyle{\invlim_\BP^i}:=R^i \invlim_\BP
$$
% exist as $\prsh(\BP)$ is abelian with enough projectives. Call
% $\invlim_\BP^* F$ 
are called
the higher limits of $F$. By definition the \emph{cohomology groups of $\BP$ with
coefficients in the presheaf $F$} are these higher limits evaluated at
$F$.

The higher limits are computed as follows. 
There is a canonical projective resolution $A_*\ra \Delta\Z$ in
$\prsh(\BP)$ with $A_n=\bigoplus_\ss
\Upsilon_{\ss_n}\Z$, 
%\marginlabel{should this be $\prod$?}
where the direct sum is over the simplicies $\ss\in N^n\BP$, and with
the maps $A_n\ra A_{n-1}$ induced by the simplicial structure of the
nerve  (see, for example, \cite[Proposition II.6.1]{Moerdijk95}). The associated
cochain 
complex $S^*(\BP;F):=\hom_{\prsh(\BP)}(A_*,F)$ thus computes the
higher limits. It has $n$th cochain group
$$
S^n(\BP; F)= \prod_{\sigma} F(\sigma_n),
$$
where the product is over the $n$-simplicies $\ss\in N^n\BP$ of the nerve. We
adopt the convention that $S^*(\varnothing;F)$
is the zero complex. If $s\in  S^n(\BP; F)$ and $\sigma\in
N^n\BP $ we write $s\cdot\sigma$ for the component of $s$
indexed by $\sigma$, so if $\sigma= \sigma_n\leq 
\cdots \leq \sigma_0 $ then $s\cdot \sigma \in F(\sigma_n)$.

The differential $d\colon S^{n-1}(\BP; F)\ra S^n(\BP; F)$ is defined
for $s\in S^{n-1}(\BP; F)$  and $\sigma \in N^n\BP$ by 
\begin{equation}
  \label{eq:1}
  ds\cdot \sigma = \sum_{i=0}^{n-1} (-1)^i s\cdot d_i\sigma + 
(-1)^n F_{\sigma_n}^{\sigma_{n-1}}  (s\cdot d_n\sigma)
\end{equation}
with the $d_i$ the face maps (\ref{eq:23}) of the nerve. % and 
% $F_{\sigma_n}^{\sigma_{n-1}}
% \colon F(\sigma_{n-1}) \ra F(\sigma_n)$ the map induced by the
% presheaf $F$ from the map $\sigma_n\leq \sigma_{n-1}$ in $\BP$.
% If $\BP$ is finite then there is a chain isomorphism $S^*(\BP;\Delta
%A)\cong  S^*(\BP;\Delta\Z)\otimes A$ (the ``$T$ version'' of this
%result is also true -- see below). In anycase
By defining 
$$
\hs^*(\BP;F):= H(S^*(\BP; F),d) 
$$
we have $ \textstyle{\invlim_\BP^*}F\cong\hs^*(\BP;F) $.

With a constant presheaf we recover the
topology:
\begin{equation}
  \label{eq:26}
\hs^*(\BP;\Delta A) \cong H^*(\geo{N^*\BP},A)  
\end{equation}
where the right hand side is the ordinary singular cohomology of the
geometrical realisation $\geo{N^*\BP}$ (see for example
\cite[Theorem 2.1]{Baclawski75}).

The complex $S^*(\BP;F)$ has a factor for each simplex in the nerve,
degenerate or not. There is a version which uses only non-degenerate
simplices:  let $T^n(\BP;F)=\prod_{\ss}F(\ss_n)$, taking the product over $N^n_\circ\BP$, 
the non-degenerate simplicies $\ss=\ss_n<\cdots <\ss_n$,
and the differential also given by (\ref{eq:1}). 
%The complexes
Then $T^*(\BP;F)$ is a quotient of $S^*(\BP;F)$ by the
  subcomplex consisting of the degenerate simplicies, itself homotopy
  equivalent to the zero complex, hence $T^*(\BP;F)$ and $S^*(\BP;F)$
  are homotopy equivalent.
% Every element
% $s\in S^*(\BP;F)$ has a unique expression
% $s=s_{\text{nd}}+s_{\text{d}}$ where $s_{\text{nd}}\cdot\ss=s\cdot\ss$
% (resp. $0$) for all $\ss$ non-degenerate
% (resp. degenerate), and vice-versa for $s_{\text{d}}$.
Explicitly, define
$f:S^*(\BP;F)\ra T^*(\BP;F)$ by $fs\cdot\ss=s\cdot\ss$ for
$\ss\in N^n_\circ\BP$, and
$g:T^*(\BP;F)\ra S^*(\BP;F)$ by 
$$
gt\cdot\ss
=
\left\{
\begin{array}{ll}
  t\cdot\ss,&\text{if }\ss\in N^n_\circ\BP,\\
  0,&\text{otherwise.}%\ss\in N^n\BP\setminus N^n_0\BP.
\end{array}
\right.
$$
Then $fg$ is the identity on $T^*(\BP;F)$. Let $h:S^n (\BP;F)\ra S^{n-1}(\BP;F)$ be
given by
%%%
%\marginlabel{Also $\sigma_{p+l}\neq \sigma_p$??}
 %%
$$
hs\cdot\ss=
\left\{
\begin{array}{ll}
  (-1)^p s\cdot s_p\ss,&\text{if
  }\ss=\ss_{n-1}\cdots\ss_{p+\ell}\ss_p\cdots\ss_p\ss_{p-1}\cdots\ss_0,\\ 
 0,&\text{otherwise},
\end{array}\right.
$$
where the $\ss_{p},\ldots,\ss_0$ are distinct, there are $\ell\geq 2$ with
$\ell$ \emph{even\/} repeats of $\ss_p\not=\ss_{p+\ell}$, and $s_p \colon N^{n-1}\BP\ra
N^n\BP$ is a  degeneracy map from (\ref{eq:23}). 
Then $h$ is a chain homotopy between $gf$ and the identity on $S^*(\BP;F)$.
% and so we have a homotopy equivalence
% $S^*(\BP;F)\simeq T^*(\BP;F)$.
Much of what we say about $S^*(\BP;F)$ holds
analogously for $T^*(\BP;F)$. We will content ourselves with pointing this
out where appropriate and leaving the details to the reader.

\subsection{Induced maps}
\label{section1:subsection2}

If $f\colon \BQ \ra\BP$ is a map of posets then there
are a number of induced maps and functors.  
\begin{description}
\item[\textendash] There is an induced map of simplicial sets
$N^*\BQ\ra N^*\BP$ sending $\ss=\ss_n\leq\cdots\leq\ss_0\in N^n\BQ$ to
$f\ss=f\ss_n\leq\cdots\leq f\ss_0\in N^n\BP$. 
\item[\textendash]  There is
an induced functor $\prsh(\BP)\ra\prsh(\BQ)$ sending $F\in\prsh(\BP)$
to $f^*F:=F\circ f$ and $\kappa:F\ra G$ to $f^*\kappa:f^*F\ra f^*G$ with
$f^*\kappa_x=\kappa_{fx}$. If $f$ is an inclusion
$\BQ\hookrightarrow\BP$ then we will just write $F$ for $f^*F$.
\item[\textendash] There is an induced map of groups $f^*\colon S^*(\BP;F)\ra S^*(\BQ;f^*F)$,
the \emph{pull-back}, defined for $s\in S^n(\BP;F)$ and $\ss\in N^n\BQ$ by 
\begin{equation}
  \label{eq:8}
  f^*s\cdot \ss = s\cdot f\ss.
\end{equation}
\end{description}

\begin{lemma}\label{section1:subsection1:lemma1} 
The pull-back $f^*$ is a chain map. If $g\colon \BR \ra \BQ$ is another poset map then 
$(fg)^* = g^*f^*: S^*(\BP,F)\ra S^*(\BR,(fg)^*F)$.
\end{lemma}

If $f:\BQ\rightarrow\BP$ is injective then (\ref{eq:8})
gives
a pull-back $T^*(\BP;F)\ra T^*(\BQ;f^*F)$ and the analogue of Lemma
\ref{section1:subsection1:lemma1}  holds. 

\begin{description}
\item[\textendash] If $f\colon \BQ\ra\BP$ is finite-to-one, i.e. for each $x\in\BP$ the
pre-image $f^{-1}x$ is a finite set, then there is an induced map of groups
$f_*\colon S^*(\BQ;f^*F) \ra S^*(\BP;F)$, the \emph{push-forward}, defined
for $s\in S^n(\BQ;f^*F)$ and $\sigma\in N^n\BP$ by
\begin{equation}
  \label{eq:14}
f_*s\cdot \sigma= \sum_{\tau\in f^{-1}\sigma} s\cdot \tau  
\end{equation}
If $f^{-1}\sigma$ is empty the right-hand
side is taken to be zero. By definition of the presheaf
$f^*F$ each element of $f^{-1}\sigma$ has associated to it the same
abelian group -- namely $F(\sigma_n)$ -- and the sum in (\ref{eq:14}) takes place in
this group. The
push-forward is not in general a chain map: for example if $f:\BQ\ra\BP$ is
injective but not surjective.
\end{description}

\begin{remark}
Our notation differs from that found in
\cite[Appendix 2]{Gabriel-Zisman67} where $f_*$ is used for the
induced functor $\prsh(\BP)\ra\prsh(\BQ)$ and $f^*$ denotes a left adjoint to $f_*$. 
\end{remark}

\begin{lemma}\label{section1:subsection1:lemma2}
If $f$ is injective then $f^*f_* = \text{id}$, so that
$f^*$ is surjective and $f_*$ is injective.   
\end{lemma}

\begin{description}
\item[\textendash]
Morphisms of presheaves also induce maps of complexes.
Let $F$ and $G$ be presheaves on $\BP$ and $\kappa\colon F\ra
G$ a natural transformation. % So for each $x\in \BP$ there is a map
% $\kappa_x\colon F(x) \ra G(x) $ and these are natural in $x$.
Then there is an induced map 
$\kappa_* \colon S^*(\BP;F) \ra S^*(\BP;G)$
defined for $s\in S^n(\BP;F)$ and 
$\sigma =\sigma_n\leq \cdots \leq \sigma_0 \in N^n\BP$ by
$$
\kappa_*s \cdot \sigma = \kappa_{\sigma_n} (s\cdot \sigma).
$$
\end{description}

\begin{lemma}\label{section1:subsection1:lemma4} 
The induced map $\kappa_*$ is a chain map and if $f\colon \BQ\ra \BP$ 
is a poset map then the following diagram commutes:
$$
\xymatrix{
S^*(\BP;F) 
%\vrule width 1.5mm height 0 mm depth 0mm
\ar[r]^{\kappa_*}  \ar[d]_{f^*} & S^*(\BP;G) \ar[d]^{f^*}  \\
S^*(\BQ;f^*F)  
%\vrule width 1.5mm height 0 mm depth 0mm
\ar[r]_{f^*\kappa_*}   &S^*(\BQ;f^*G) 
}
$$
\end{lemma}

\subsection{Reduced cohomology}
\label{section1:subsection4}

For $A\in\ab$ we can augment $S^*(\BP;\Delta A)$ in degree $-1$ by defining
$d^{-1}\colon A\ra
S^0(\BP;\Delta A)$ to be $d^{-1}a\cdot\ss=a$ for $\ss\in N^0\BP$
(i.e. $d^{-1}$ injects $A$ diagonally). Then
$\widetilde{S}^*(\BP;\Delta A):=A\ra S^*(\BP;\Delta A)$
is a cochain complex, and the \emph{reduced\/} cohomology is defined by
$$
\widetilde{\hs}^*(\BP;\Delta A):=
H (\widetilde{S}^*(\BP;\Delta A),d).
% \left\{
%   \begin{array}{ll}
%     A,&\text{if }\BP=\varnothing\text{ or }
%   \end{array}
% \right.
$$
This is isomorphic to $\widetilde{H}^*(\geo{N^*\BP};A)$,
the ordinary reduced cohomology of the realization of the nerve.
Sometimes it will be convenient to set
$N^{-1}\BP=N^{-1}_\circ\BP=\star$, the one element set, and define $a\cdot\star=a$
for $a\in A$.

An alternative construction is as follows: let
$\star$ be the one-element poset and let 
$F\in\prsh(\star)$ be the presheaf with $F(\star) = A$.  The collapse map
$f\colon \BP \ra \star$ induces the constant presheaf $\Delta A$ on
$\BP$ via $f^*F = \Delta A$. By considering the map 
$f^*\colon \hs^*(\star;F) \ra \hs^*(\BP;\Delta A)$ induced by the
pullback, we have
$$
\widetilde{\hs}^* (\BP;\Delta A) \cong \coker f^*.
$$
Analogously we can define $\widetilde{T}^*(\BP;\Delta A)$ and $\widetilde{\htt}^*(\BP;\Delta A)$
-- although only the first of the two approaches above now works -- and we
have a homotopy
equivalence $\widetilde{T}^*(\BP;\Delta
A)\simeq\widetilde{S}^*(\BP;\Delta A)$.

\subsection{Relative cohomology}
\label{section1:subsection3}

Let $f\colon \BQ\ra\BP$ be a
poset map, $F$ a presheaf on $\BP$ and $f^*\colon  S^*(\BP;F)
\ra S^*(\BQ;f^*F)$ the pull-back. Define
$$
S^*(\BP,\BQ;F):= \Ker f^*.
$$ 
We will mostly consider the case where $f$ is an inclusion, which is why we omit it
from the notation. Observe that for $s\in S^*(\BP;F)$ we have  $s\in
S^*(\BP,\BQ;F)$ if and only if  $s\cdot\ss=0$ for all
$\ss\in f(N^*\BQ)\subset N^*\BP$.
The differential on $S^*(\BP;F)$ restricts to a differential on $S^*(\BP,\BQ;F)$ and we define
$$
\hs^*(\BP,\BQ;F) := H(S^*(\BP,\BQ;F), d),
$$
the (\emph{relative}) cohomology of the pair $(\BP,\BQ)$ with coefficients in the
presheaf $F$. We adopt the convention that  $S^*(\BP,\varnothing;F)=
S^*(\BP;F)$. If $f$ is an injection we can
analogously define $T^*(\BP,\BQ;F)$. The maps given at the end of
\S\ref{section1:subsection1} then restrict to the various relative complexes to give a homotopy
equivalence $T^*(\BP,\BQ;F)\simeq S^*(\BP,\BQ;F)$.

If $f$ is injective then Lemma \ref{section1:subsection1:lemma2} gives
a short exact sequence 
$$
0\longrightarrow 
S^*(\BP,\BQ;F)
\longrightarrow
S^*(\BP;F)
\stackrel{f^*}{\longrightarrow}
S^*(\BQ;f^*F)
\longrightarrow
0
$$
and hence a long exact sequence
\begin{equation}
  \label{eq:6}
\cdots
\stackrel{\beta}{\longrightarrow}
\hs^n(\BP,\BQ;F)
\longrightarrow
\hs^n(\BP;F)
\longrightarrow
\hs^n(\BQ;f^*F)
\stackrel{\beta}{\longrightarrow}
\hs^{n+1}(\BP,\BQ;F)
\longrightarrow
\cdots  
\end{equation}

\begin{lemma}\label{lem:pqr}
Let  $\BP, \BQ , \BR$ be posets with $j\colon
\BR\hookrightarrow \BQ$ and $i\colon\BQ\hookrightarrow\BP$ inclusions
and 
let $F$ be a presheaf on $\BP$.  Then there is a short exact
sequence
$$
0\longrightarrow 
S^*(\BP,\BQ;F)
\longrightarrow
S^*(\BP,\BR;F)
\stackrel{i^*}{\longrightarrow}
S^*(\BQ,\BR;F)
\longrightarrow
0
$$
and hence a long exact sequence
$$
\cdots
\stackrel{\delta}{\longrightarrow}
\hs^n(\BP,\BQ;F)
\longrightarrow
\hs^n(\BP,\BR;F)
\longrightarrow
\hs^n(\BQ,\BR;F)
\stackrel{\delta}{\longrightarrow}
\hs^{n+1}(\BP,\BQ;F)
\longrightarrow
\cdots
$$
\end{lemma}

%\begin{proof}
%A composition of surjections of complexes 
%$A^*\stackrel{f}{\twoheadrightarrow}B^*\stackrel{g}{\twoheadrightarrow}C^*$
%induces a short exact sequence
%$0\longrightarrow 
%\ker(f)
%\longrightarrow
%\ker(gf)
%\longrightarrow
%\ker(g)
%\longrightarrow
%0$.
%We take $f=i^*$ and $g=j^*$.
%\qed 
%\end{proof}

\begin{lemma}\label{lem:b}
Let $(\BP,\BQ,\BR)$ be the triple of Lemma \ref{lem:pqr} with
$\iota\colon \hs^{n}(\BQ,\BR; F) \ra \hs^n(\BQ; F)$ induced by the
inclusion $S^{n}(\BQ,\BR; F) \ra S^n(\BQ;F)$ and 
$\beta:HS^n(\BQ;F)\ra HS^{n+1}(\BP,\BQ;F)$ the
connecting homomorphism of the pair
$(\BP,\BQ)$. If $\delta$ is the connecting homomorphism of Lemma
\ref{lem:pqr} then $\delta= \beta \iota$.
\end{lemma}

\section{Cellular cohomology of posets with
  coefficients in a presheaf}
\label{section2}

Singular cohomology can be defined for any space
$X$. If $X$ has a cellular structure -- for example is a
CW-complex -- then cellular cohomology can also be defined and singular
cohomology can be computed using it. In this section we define, for a
large class of posets, a cellular cohomology with
coefficients in a presheaf. % We then describe the complex used in the
% definition in a number of ways.
A primordial version, along the lines of Proposition \ref{proposition:cellular:relative}
    below, appears in \cite[Section 4]{Yuzvinsky87}.
As in topology we define a cochain complex using the relative
cohomology of pairs -- which is not hard -- and then expend some effort
describing it more explicitly. 
In the next section we will describe a
situation in which the $\invlim_\BP^* F$ can be computed via this
cellular cohomology. 
% This section and the next develop analogous tools for
% $\invlim_\BP^* F=HS^*(\BP;F)$. 

\subsection{The definition of cellular cohomology}
\label{section2:subsection1}

We begin by recalling some poset terminology. If $x\leq y\in\BP$ and for any
$x\leq z\leq y$ we have either $z=x$ or $z=y$, then 
$y$ is said to \emph{cover\/} $x$, written $x\prec y$. % The poset $\BP$
% is \emph{connected\/} if for any $x,y\in\BP$ there are
% $x_0,\ldots,x_n$ with $x_0=x,x_n=y$ and $x_i,x_{i+1}$ comparable
% (i.e. either $x_i\leq x_{i+1}$ or $x_{i+1}\leq x_i$).
$\BP$ is
\emph{graded\/} if there exists a \emph{rank function\/} $\rk:\BP\ra\Z$, i.e. a
function such
that (i) $x<y$ implies $\rk(x)<\rk(y)$, and (ii) $x\prec y$
implies $\rk(y)=\rk(x)+1$. % If $\rk,\rk'$ are two rank functions on
% $\BP$ then $\rk-\rk'$ is constant on the connected components of $\BP$;
% in particular, any two rank functions on a connected $\BP$ differ by a
% translation.
% A graded $\BP$ will have many rank functions but any two differ only
% by a translation on each connected component. 

Fix a rank function $\rk$ on graded $\BP$ and assume further that
$\rk$ is bounded above with $r=\max_{x\in\BP}\{\rk (x)\}$. Define the
\emph{corank\/} function 
$\cork{\cdot} : \BP \ra \Z^{\geq 0}$
by $\cork{x} = r-\rk (x)$. Filter $\BP$ by corank by letting
\begin{equation}
  \label{eq:24}
\BPn k = \{x\in \BP \mid \cork{x} \leq k\}.  
\end{equation}
There is thus a sequence of inclusions
$\BPn 0 \subset \BPn 1 \subset \BPn 2 \subset \cdots$
If $F$ is a presheaf on $\BP$ then we get induced presheaves $F$ on
each $\BP^k$ via the inclusions $\BP^k\hookrightarrow \BP$.
% and the presheaf $F\colon \BPop \ra \ab $ defines a presheaf (also denoted
% $F$) on each $\BPn k$.

\emph{From now on all posets will be graded with a bounded rank function
and with a corank function $\cork{\cdot}: \BP \ra \Z^{\geq 0}$;
we will
abbreviate this set-up to ``graded with a corank function''.}

\begin{definition}
Let $\BP$ be graded with a corank function, $\{\BP^k\}_{k\in\Z}$
the associated filtration in (\ref{eq:24}) and $F$ a presheaf on $\BP$.
%$F\in \prsh(\BP)$.
The cellular cochain complex $C^*(\BP;F)$ has chain groups
$$
C^n(\BP;F) := \hs^n(\BPn n , \BPn {n-1} ; F)
$$
and differential 
%$\delta \colon C^{n-1} \ra C^{n}$ 
$C^{n-1}(\BP;F)\kern-1pt\stackrel{\delta}{\longrightarrow}C^n(\BP;F)$
by taking $\BR=\BPn {n-2}$, $\BQ=\BPn {n-1}$ and $\BP=\BPn
{n}$ in Lemma \ref{lem:pqr} and defining $\delta$ to be the coboundary map in
the associated the long exact
sequence
$$
C^{n-1}(\BP;F) = \hs^{n-1}(\BPn {n-1},\BPn {n-2};F)
\stackrel{\delta}{\longrightarrow}
\hs^{n}(\BPn {n},\BPn {n-1};F) = C^{n}(\BP;F).
$$
The cellular cohomology
  of $\BP$ with coefficients in the presheaf $F$ is defined to be the
homology of this complex:
$$
\hc^*(\BP; F) := H(C^*(\BP;F),\delta).
$$
\end{definition}

That $\delta^2=0$ is a standard argument following from Lemma
\ref{lem:b}. Clearly $C^n(\BP;F)=0$ for $n<0$ and in degree zero we have 
$$
C^0 (\BP;F)= \hs^0(\BPn 0 , \varnothing ; F)= \hs^0(\BPn 0 ; F) = 
{\textstyle{\invlim_{\BP^0} F}}=\prod_{\cork{x} = 0} F(x).
$$

% We have $T^i(\BP^n;F)=0$ for $i>n$, so that
% \marginlabel{This looks misplaced: where do we need it?}
% \begin{equation}
%   \label{eq:15}
% \hs^i(\BP^n,\BP^{n-1};F)=0\text{ for }i>n.
% \end{equation}

We can write an explicit formula for $\delta$: if $s\in S^{n-1}(\BP^{n-1},\BP^{n-2};F)$ is a
cocycle with homology class $[s]$ then $\delta [s]=[t]$ where for $\ss\in N^n\BP^n$ we have
\begin{equation}
  \label{eq:4}
t\cdot\ss=
\left\{
  \begin{array}{ll}
    (-1)^n F^{\ss_{n-1}}_{\ss_n}(s\cdot d_n\ss),&\text{ if
    }\cork{\ss_n}=n\text{ and }\cork{\ss_{n-1}}=n-1,\\
    0,&\text{otherwise.}
  \end{array}
\right.  
\end{equation}
This comes about by
writing a formula for the connecting homomorphism of the pair
$(\BP^n,\BP^{n-1})$ and then using Lemma \ref{lem:b}.

\subsection{Describing the cellular cochain complex}
\label{section2:subsection2}

Continuing the analogy with the cellular cohomology of a CW-complex, the role of cells is
played by the closed intervals $\BPge{x}$ and the role of boundary spheres of
cells by
the open intervals $\BPg{x}$:
$$
\BPge{x}=\{y\in\BP\,|\,y\geq x\}\text{ and }
\BPg{x}=\{y\in\BP\,|\,y> x\}.
$$
$\BP$ is \emph{locally finite\/} if for any $x\in\BP$ there are only finitely many
$y$ with $x\prec y$. If $\BP$ is graded with a corank function and
locally finite, then the interval $\BPge{x}$ is a finite poset for
each $x$. % In a similar way to (\ref{eq:15}) we also have
% $\hs^i(\BPge{x},\BPg{x};F)=0$ for $i>\cork{x}$.

We now describe the cellular
cochain complex $C^*(\BP;F)$ in terms of these intervals. The
exposition is complicated slightly by the fact that there may be
infinitely many elements of a given corank. In any case,
let $x$ be a fixed element of corank $n$. Then the diagram of
inclusions below left induces, by Lemmas
\ref{section1:subsection1:lemma1} and
\ref{section1:subsection1:lemma2}, the commuting diagram on the right: 
$$
\xymatrix{
\BPge{x} \vrule width 1.5mm height 0 mm depth 0mm
\ar@{>->}[r]^-{\ve_x}  
& 
\BPn{n}\\
\BPg{x}  \vrule width 1.5mm height 0 mm depth 0mm
\vrule width 0mm height 4.5mm depth 0mm
\ar@{>->}[u]^{j} \ar@{>->}[r]   
&
\BPn{n-1} \vrule width 0mm height 5mm depth 0mm
\ar@{>->}[u]_{i}
}
\;\;\;\;\;\;\;\;\;\;
\xymatrix{
S^*(\BPge{x};F) 
\ar@{->>}[d]_-{j^*} & 
S^*(\BPn{n};F) \ar@{->>}[d]^-{i^*}  \ar@{->>}[l]_-{\ve_x^*}\\
S^*(\BPg{x};F)  
&S^*(\BPn{n-1};F) \ar@{->>}[l]   
}
$$
As the diagram commutes, $\ve_x^*$ restricts to a chain map $\ker
i^*\ra \ker j^*$, i.e. a chain map
$$
\ve_x^*:S^*(\BP^n,\BP^{n-1};F)\ra S^*(\BPge{x},\BPg{x};F)
$$
which for $s\in S^*(\BP^n,\BP^{n-1};F)$ and $\ss\in N^*\BPge{x}$ is
given by $\ve_x^* s\cdot\ss=s\cdot \ve_x\ss$. 

\begin{proposition}\label{lem:epsilonchain}
The map of abelian groups
\begin{equation}
  \label{eq:16}
\ve: S^*(\BP^n,\BP^{n-1};F)\ra \prod_{\cork{x}=n} S^*(\BPge{x},\BPg{x};F)  
\end{equation}
given by
$\ve s\cdot x=\ve_x^*s$ is
a chain isomorphism.
\end{proposition}

\begin{proof}
The differential on the product is $\prod d_x$ with $d_x$ the
differential on $S^*(\BPge{x},\BPg{x};F)$, and $\ve$ is a chain map as
each $\ve^*_x$ is.
Let $s,s'\in S^i(\BP^n,\BP^{n-1};F)$ be such that $\ve s=\ve
s'$. If $\ss\in N^i\BP^n\setminus N^i\BP^{n-1}$ then
$|\ss_i|=n$ and so $\ss=\ve_{\ss_i}(\ss)$
for $\ss\in N^i\BPge{\ss_i}$. In particular $s\cdot\ss=s\cdot
\ve_{\ss_i}(\ss)=s'\cdot \ve_{\ss_i}(\ss)=s'\cdot\ss$ and so $\ve$ is injective.
On the otherhand let $s\in \prod_{\cork{x}=n} S^*(\BPge{x},\BPg{x};F)$ be of degree $i$.
If $\tau\in N^i\BP^n\setminus N^i\BP^{n-1}$ then
$\tau=\ve_{\tau_i}(\tau)$ with $\tau\in N^i\BP_{\geq\tau_i}$ but is not in the 
image of any other $\ve_x$. Let $t\in S^i(\BP^n;F)$ be such
that
$$
t\cdot\tau=
\left\{
    \begin{array}{ll}
      (s\cdot x)\cdot\ss,&\text{if }\tau\in N^i\BP^n\setminus
    N^i\BP^{n-1}\text{ where }\tau=\ve_x(\tau)\\
      0,&\tau\in N^i\BP^{n-1}.
    \end{array}
\right.
$$
Then $t\in S^i(\BP^n,\BP^{n-1};F)$ with $\ve t=s$ and so $\ve$ is surjective. %\qed 
\end{proof}

Writing $\ve$ as well for the composition
$$
\hs^*(\BP^n,\BP^{n-1};F)
\stackrel{\cong}{\longrightarrow} 
H\biggl(\prod_{\cork{x}=n}S^*(\BPge{x},\BPg{x};F)\biggr)
\stackrel{\cong}{\longrightarrow} 
\prod_{\cork{x}=n}\hs^*(\BPge{x},\BPg{x};F)
$$
the differential of the cellular cochain complex can be described in 
terms of the isomorphism (\ref{eq:16}) as the map making the following diagram commute:
$$
\xymatrix{
C^{n-1}=\hs^{n-1}(\BP^{n-1},\BP^{n-2};F) 
\ar@{->}[d]_-{\ve} \ar@{->}[r]^-{\delta}& 
\hs^{n}(\BP^{n},\BP^{n-1};F)=C^n  \ar@{->}[d]^-{\ve}\\
\prod_{\cork{y}=n-1}\hs^{n-1}(\BPge{y},\BPg{y};F) \ar@{.>}[r] &
\prod_{\cork{x}=n}\hs^n(\BPge{x},\BPg{x};F) 
}
$$
We will call this map $\delta$ as well.
An explicit formula for $\ve^{-1}$ can be extracted from the surjectivity part
of the proof of Proposition \ref{lem:epsilonchain} and by combining
this with
(\ref{eq:4}) and (\ref{eq:16}) we have proved the following alternative
description of the cellular cochain complex:

\begin{proposition}\label{cellular:complex:version2}
Let $\BP$ be graded with a corank function and $F$ a presheaf on
$\BP$. Then there are isomorphisms 
$$
C^n(\BP;F) \cong \prod_{\cork{x} = n}\hs^n(\BPge x, \BPg x ; F)  
$$
with respect to which the differential in the cellular cochain complex  
$C^{n-1}(\BP;F) \kern-1pt\stackrel{\delta}{\longrightarrow}C^n(\BP;F)$
%$\delta \colon C^{n-1} \ra C^{n}$,
has the following effect on an element $s\in\prod_{\cork{y}=n-1} \hs^{n-1}(\BPge y,\BPg
y;F)$.  Suppose that $s\cdot y=[s_y]$ for $s_y$ a cocycle in $S^{n-1}(\BPge
y,\BPg y;F)$. Then $\delta s\cdot x=[t_x]$ where
$t_x\in S^{n}(\BPge x,\BPg x;F)$ is given by
\begin{equation}
  \label{eq:2}
  t_x\cdot\ss=
\left\{
  \begin{array}{ll}
   (-1)^n F^{y}_{x}(s_y\cdot d_n\ss),&\text{if }\ss_n=x\prec y=\ss_{n-1},\\
   0,&\text{otherwise,} 
  \end{array}
\right.
\end{equation}
where $\ss\in N^n\BP_{\geq x}$.
\end{proposition}

%There is a $T$-version of this description which is identical except
%that $s_y$ can be an arbitrary element in
%$T^{n-1}(\BPge{y},\BPg{y};F)$ -- every element is a cocycle! -- and $\ss\in N^n_\circ\BPge{x}$.

In particular for a fixed element $x$ of corank $n$, the diagram
\begin{equation}
  \label{eq:3}
\xymatrix{
C^{n-1}\cong\prod_{\cork{y}=n-1}\hs^{n-1}(\BPge{y},\BPg{y};F)\ar@{->}[r]^-{\delta} \ar@{->}[d]_-{\text{proj.}} &
\prod_{\cork{x}=n}\hs^n(\BPge{x},\BPg{x};F)\cong C^n
\ar@{->}[d]^-{\text{proj.}}\\
\prod_{x\prec y}\hs^{n-1}(\BPge{y},\BPg{y};F) \ar@{->}[r]& \hs^{n}(\BPge{x},\BPg{x};F)
}  
\end{equation}
commutes, where the bottom horizontal map is the restriction of
$\delta$ to the $y$ covering $x$
followed by projection onto the $x$-coordinate.
In words, if $s\in\prod_{\cork{y}=n-1} \hs^{n-1}(\BPge y,\BPg
y;F)$ then the component of $\delta s$ indexed by $x$ depends
only on the components of $s$ indexed by the $y$ covering $x$.

If $y$ is a fixed element of corank $n-1$ and $x$ a fixed element
of corank $n$, then the map $\delta^y_x$ making the diagram
$$
\xymatrix{
C^{n-1}\cong\prod_{\cork{y}=n-1}\hs^{n-1}(\BPge{y},\BPg{y};F)\ar@{->}[r]^-{\delta} \ar@{->}[d]_-{\text{proj.}} &
\prod_{\cork{x}=n}\hs^n(\BPge{x},\BPg{x};F)\cong C^n
\ar@{->}[d]^-{\text{proj.}}\\
\hs^{n-1}(\BPge{y},\BPg{y};F) \ar@{.>}[r]^-{\delta^y_x}& \hs^{n}(\BPge{x},\BPg{x};F)
}
$$
commute is called the \emph{matrix element\/} corresponding to the
pair $(x,y)$. Explicitly, if $s$ is a cocycle in
$\hs^{n-1}(\BPge{y},\BPg{y};F)$ then $\delta^y_x[s]=[t]$ where for
$\ss\in N^n\BPge{x}$, the coordinate $t\cdot\ss$ is given by
(\ref{eq:2}) with $s_y$ replaced by $s$. In general $\delta$ is
not determined by its matrix elements. If $\BP$ is locally
  finite however then the bottom left term in 
(\ref{eq:3}) is a direct sum, and for $s\in C^{n-1}$ we have 
$$
\delta s\cdot x=\sum_{x\prec y} \delta^y_x(s\cdot y).
$$

% Now we replace the relative
% homology by the homology of the posets 
% $\BPg x$ with a constant coefficient presheaf.

We can further refine the chain groups of the cellular complex:

\begin{proposition}\label{section2:subsection2:proposition2}
Let $x\in\BP$. Then %have corank $n\geq 0$. Then 
$$
\hs^*(\BPge x, \BPg x;F) 
\cong \hs^*(\BPge x, \BPg x;\Delta F(x)) 
\cong \widetilde{\hs}^{*-1}(\BPg x ; \Delta F(x))
$$ 
where $\Delta F(x)$ is the constant presheaf with value $F(x)\in\ab$. 
\end{proposition}

\begin{proof}
Consider $F$ and $\Delta F(x)$ as presheaves on the closed interval $\BPge{x}$.
In $\prsh(\BPge{x})$ there is a natural transformation $\kappa\colon F
\ra \Delta F(x)$ defined by $\kappa_y = F^y_x$
which induces a chain map 
$\kappa_* \colon S^*(\BPge x ; F) \ra S^*(\BPge x ;\Delta F(x))$.
The inclusion $\BPg{x}\hookrightarrow\BPge{x}$ and 
Lemma \ref{section1:subsection1:lemma4} mean that $\kappa_*$
restricts to a chain map
$\kappa_* \colon S^*(\BPge x,\BPg x ; F) \ra S^*(\BPge x,\BPg x ;\Delta F(x))$
which turns out to be an inclusion. 
Surjectivity is also not hard to show,
and we have an induced isomorphism in cohomology:
$$
\hs^*(\BPge x,\BPg x ; F) 
\stackrel{\cong}{\longrightarrow} 
\hs^*(\BPge x,\BPg x ;\Delta F(x)).
$$
We observed in
\S\ref{section1:subsection1} that $\hs^*(\BPge{x};\Delta F(x))\cong
H^*(\geo{N^*\BPge{x}},F(x))$, the ordinary singular cohomology of the
classifying space $\geo{N^*\BPge{x}}$. As $\BPge{x}$ has a unique minimal
element the space $\geo{N^*\BPge{x}}$ is contractible, giving
$$
\hs^i(\BPge x;\Delta F(x)) \cong 
\begin{cases}
  F(x), & i=0,\\
0, & \text{ otherwise.}
\end{cases}
$$
Thus for $i>1$ the coboundary map 
in the  long exact sequence (\ref{eq:6}) of the pair
$(\BPge{x},\BPg{x})$ is an isomorphism
$$
\widetilde{\hs}^{i-1}(\BP_{>x};\Delta F(x))
\cong
\hs^{i-1}(\BP_{> x};\Delta F(x))
\stackrel{\cong}{\longrightarrow}
\hs^i(\BP_{\geq x},\BP_{>x};\Delta F(x)).
$$
$\hs^0(\BP_{\geq x};\Delta F(x))$ is the diagonal copy of $F(x)$
in $S^0(\BP_{\geq x};\Delta F(x))$, so
$\hs^0(\BP_{\geq x},\BP_{>x};\Delta F(x))$ is trivial when $\cork{x}>0$ and
isomorphic to $F(x)$ when $\cork{x}=0$ (and thus $\cong
\widetilde{\hs}^{-1}(\BP_{>x};\Delta F(x))$ in either case). The long
exact sequence of the pair $(\BP_{\geq x},\BP_{>x})$
collapses to the short exact sequence:
$$
0
\ra HS^0(\BPge{x};\Delta F(x)) 
\ra HS^0(\BPg x;\Delta F(x)) 
\ra HS^1(\BPge x,\BPg x;\Delta F(x)) 
\ra 0.
$$
When $\cork{x}>0$ the first map can be identified with 
$\hs^0(\star;\Delta F(x))\ra \hs^0(\BP_{>x};\Delta F(x))$, giving
$\hs^1(\BP_{\geq x},\BP_{>x};\Delta
F(x))\cong\widetilde{\hs}^0(\BP_{>x};\Delta F(x))$. For $\cork{x}=0$ these
two are both trivial. %\qed
\end{proof}

This leads to our third description of the cellular chain complex in
terms of the reduced cohomology of open intervals. 
Writing $\aa$ for the isomorphism
$$
\prod_{\cork{x}=n}\widetilde{\hs}^{n-1} (\BPg{x};\Delta F(x))
\stackrel{\cong}{\longrightarrow}
\prod_{\cork{x}=n} \hs^n (\BPge{x},\BPg{x};F),
$$
induced by Proposition \ref{section2:subsection2:proposition2},
then if
$s$ is an element of the left hand side with $s\cdot x=[s_x]$ for $s_x$ a cocycle in
$\widetilde{S}^{n-1}(\BPg{x};\Delta F(x))$, we have $\aa s=t$ with $t\cdot
x=[t_x]$, where $t_x\in S^n(\BPge{x},\BPg{x};\Delta F(x))$ is given by
$$
t_x\cdot\ss=
\left\{
  \begin{array}{ll}
    (-1)^n s_x\cdot d_n\ss,&\text{if }\ss_n=x<\ss_{n-1},\\
    0&\text{otherwise},
  \end{array}
\right.
$$
and $\ss\in N^n\BPge{x}$.
Computing $\aa^{-1}\delta\aa$ with the $\delta$ of (\ref{eq:2})
gives:

\begin{proposition}
\label{proposition:cellular:relative}
Let $\BP$ be graded with a corank function and $F$ a presheaf on $\BP$.
Then there are isomorphisms
$$
C^n \cong \prod_{\cork{x} = n}\widetilde{\hs}^{n-1}(\BPg x ; \Delta F(x))  
$$
and differential 
$C^{n-1}\stackrel{\delta}{\longrightarrow} C^n$
%$\delta \colon C^{n-1} \ra C^{n}$, 
where if $s\in\prod_{\cork{y}=n-1}\widetilde{\hs}^{n-2}(\BPg{y};\Delta
F(y))$  
with $s\cdot y=[s_y]$ \/for $s_y$ a cocycle in $\widetilde{S}^{n-2}(\BPg y;\Delta F(y))$, then
$\delta s=t$ with $t\cdot x=[t_x]$ where 
$t_x\in\widetilde{S}^{n-1}(\BPg x;\Delta F(x))$ is given by
\begin{equation}
  \label{eq:5}
t_x\cdot\ss=
\left\{
  \begin{array}{ll}
    (-1)^{n-1} F^y_x(s_y\cdot d_{n-1}\ss),&\text{if }x\prec y=\ss_{n-1}<\ss_{n-2},\\
    0&\text{otherwise,}
  \end{array}
\right.  
\end{equation}
and $\ss\in N^{n-1}\BPg{x}$. 
\end{proposition}

Similar comments pertain to this description of the cellular complex
as those following Proposition \ref{cellular:complex:version2}. In particular if $y$ is fixed of
corank $n-1$ and $x$ fixed of corank $n$ then the matrix element
$\delta^y_x$ sends $[s]$ to $[t]$ with $s\in 
\widetilde{\hs}^{n-2}(\BPg{y},\Delta F(y))$ a cocycle and
$t\cdot\ss$ given by (\ref{eq:5}) with $s_y$ replaced by $s$.

%\subsection{Generators and  the cellular chain complex} 
\subsection{Locally finite posets} 
\label{section2:subsection3}

When
$\BP$ is locally finite the middle of the three terms in Proposition
\ref{section2:subsection2:proposition2} can be further reduced. 
For $\BPge{x}$ is then a finite poset, so that
$S^*(\BPge{x},\BPg{x};\Delta\Z)$ has free cochain groups and hence
$S^*(\BPge{x},\BPg{x};\Delta F(x))\cong
S^*(\BPge{x},\BPg{x};\Delta\Z) \otimes F(x)$. In any case -- locally
finite or not -- $HS^i(\BPge{x},\BPg{x};F)$ $\cong
HT^i(\BPge{x},\BPg{x};F)=0$ for $i>\cork{x}$ and so % universal
% coefficients gives
\begin{equation}
  \label{eq:25}
\hs^n(\BPge{x},\BPg{x};\Delta F(x))
\cong
HT^n(\BPge{x},\BPg{x};\Delta\Z)\otimes F(x)  
\end{equation}
when $\cork{x}=n$,
giving
$$
C^n\cong
\prod_{\cork{x}=n} \htt^n(\BPge{x},\BPg{x};\Delta\Z)\otimes F(x).
$$
The differential
%$C^{n-1}\kern-1pt\stackrel{\delta}{\longrightarrow}C^n$ 
is determined
by the matrix elements $\delta_x^y$ in the locally finite case, and the image of an $s\in
\htt^{n-1}(\BPge{y},\BPg{y};\Delta\Z)\otimes F(y)$ is determined by
the images of elements of the form $[s_\ss]\otimes a$, where $a\in
F(y)$, and for $\ss\in N_\circ^{n-1}\BPge{y}$ (with necessarily
$\ss_{n-1}=y$) the tuple $s_\ss\in
T^{n-1}(\BPge{y},\BPg{y};\Delta\Z)$ has $1$ in the $\ss$-coordinate
and $0$'s elsewhere. Then
\begin{equation}
  \label{eq:20}
\delta_x^y:[s_\ss]\otimes a\mapsto [s_{x\ss}]\otimes (-1)^n F_x^y(a) 
\end{equation}
where $x\ss$ is the result of pre-appending $x$ to $\ss$. 

Similarly we have
$$
\widetilde{\hs}^{n-1}(\BPg{x};\Delta F(x))
\cong
\widetilde{HT}^{n-1}(\BPg{x};\Delta\Z)\otimes F(x)
$$
and so 
\begin{equation}
  \label{eq:18}
C^n\cong
\prod_{\cork{x}=n}
\widetilde{\htt}^{n-1}(\BPg{x};\Delta\Z)\otimes F(x)  
\end{equation}
% with differential
% $C^{n-1}\kern-1pt\stackrel{\delta}{\longrightarrow}C^n$, where if 
% $s\in \prod_{\cork{y}=n-1}
% \widetilde{\htt}^{n-2}(\BPg{y};\Delta\Z)\otimes F(y)$ with $s\cdot
% y$ a sum of elements of the form $[s_y]\otimes a_y$ for $s_y\in
% \widetilde{T}^{n-2}(\BPg{y};\Delta\Z)$ and $a_y\in F(y)$, then $\delta
% s\cdot x$ is a sum of elements of the form
with
\begin{equation}
 \label{eq:10}
\delta_x^y: [s_\ss]\otimes a\mapsto [s_{y\ss}]\otimes (-1)^{n-1}F_x^y(a),
\end{equation}
where $\ss\in N_\circ^{n-2}\BPg{y}$ and $s_\ss,a$ are analogous to (\ref{eq:20}).
% the sum over the $\ss\in N^{n-1}_\circ\BPge{x}$ with
% $x\prec y=\ss_{n-1}<\ss_{n-2}$, and
% $e_\ss\in\widetilde{T}^{n-1}(\BPg{x};\Delta\Z)$ as in the comments 
% following (\ref{eq:9}).
%{\blue [check the  sign in (\ref{eq:10})]}

We now turn to a description of $C^*(\BP;F)$ in terms of generators and relations.
If $X$ is a CW-complex then the cellular cochains in degree $n$ are
free on the $n$-cells -- one chooses free generators by fixing
orientations for the cells in the usual way. Looking forward to
\S\ref{section4:2}, if $\BP$
is the cell poset of a regular CW-complex then $C^n$ will turn out --
not surprisingly -- to be free on the elements of corank $n$, and free generators
can easily be found. If $\BP$ is a geometric lattice (see
\S\ref{section4:5}) then $C^n$ is also
free, but with
each $x$ of corank $n$ now contributing $(-1)^n\mu(x,\1)$ free factors,
rather than just one, where
$\mu$ is the M\"obius function of the lattice\footnote{The M\"obius
  function is the inverse of the zeta function in the incidence
  algebra of $\BP$. Explicitly, for $k$ a field it is the $k$-valued function on
the intervals given by $\mu(x,x)=1$ and $\mu(x,y)=-\sum_{x\leq
  z<y}\mu(x,z)$ when $x<y$. See \cite[Section 3.7]{Stanley12}}.  
In general the situation is more complex. First of all $C^n$ may not
be free, and even when it is
we prefer not to privelege a particular choice of basis. 

We thus give a non-free presentation for $C^n$ when $\BP$ is locally
finite (see
also \cite[\S1.6]{Wachs07}). Let
$x\in \BP$ be fixed with $\cork{x}=n$ and recall that 
$N^n_\circ\BPge{x}$ are the non-degenerate $n$-simplicies in the nerve
of $\BPge{x}$. The
elements of $N^n_\circ\BPge{x}$ thus have the
form $\ss=x\ss_{n-1}\cdots\ss_0$ with $\cork{\ss_i}=i$. 

\begin{definition}
Let $\tau=x\tau_{n-1}\cdots\tau_{j+1}\tau_{j-1}\cdots\tau_0$ be a
fixed $(n-1)$-simplex in
$N^{n-1}_\circ\BPge{x}$ with $0\leq j<n$ and $\cork{\tau_i}=i$. 
We call the set $B_\tau$ of all $n$-simplicies in $N^n_\circ\BPge{x}$ of the form
$$
x\tau_{n-1}\cdots\tau_{j+1}y\,\tau_{j-1}\cdots\tau_0
$$ 
(where necessarily $\cork{y}=j$) the compatible family given by $\tau$. 
% Let $B$ be a compatible family and let 
% $
% s\in \bigoplus_{\sigma \in T_x} F(x)
% $.
Let $s\in\bigoplus_\ss F(x)$, the direct sum over the $\ss\in
N^n_\circ\BPge{x}$, and let $B_\tau$ be some compatible family. 
We say that $s$ is $B_\tau$-constant if there is a fixed $a\in F(x)$ such
that %for $\ss\in N^n_\circ\BPge{x}$ 
$$
s\cdot\ss =
\left\{
\begin{array}{ll}
 a, & \ss\in B_\tau,\\
 0, &\text{ otherwise.}  
\end{array}\right.
$$
\end{definition}

Now let
$K_x\subset\bigoplus_{\ss}  F(x)$
be the subgroup generated by the $s$ that are $B_\tau$-constant for
some $\tau$, where
$\tau$ ranges over the $(n-1)$-simplicies of $N^{n-1}_\circ\BPg{x}$.

\begin{proposition}\label{prop:Kiso}
Let $\BP$ be locally finite and let $x\in\BP$ with $\cork{x}=n$. Then
there are isomorphisms
$$
%\hs^{n-1}(\BPg x;\Delta F(x)) 
\hs^n(\BPge{x},\BPg{x};\Delta F(x))
\cong
\biggl(\bigoplus_{\ss} F(x)\biggr) /K_x
\cong
A_x\otimes F(x)
$$ 
where the direct sum is over the $\ss\in N^n_\circ\BPge{x}$, and $A_x$
is the abelian group 
having presentation with generators the $\ss\in N^n_\circ \BPge{x}$ and
relations the $\sum_{\sigma \in B_\tau} \sigma = 0$,
for each compatible family $B_\tau$.
\end{proposition}

% Here $(\bigoplus_{\ss} \Z) /K_x$ is the abelian group 
% having presentation with generators the $N^n_\circ \BPge{x}$ and
% relations the $\sum_{\sigma \in B_\tau} \sigma = 0$
% for each compatible family $B_\tau$.

\begin{proof}
We have $T^n(\BPge{x},\BPg{x};\Delta
F(x))$ is the direct sum $\bigoplus_\ss F(x)$ over the $\ss\in
N^n_\circ\BPge{x}$. Moreover $T^{n+1}=0$ so that
$\hs^n\cong\htt^n=(\bigoplus_\ss F(x))/\im d^{n-1}$, and it remains to
show that $\im d^{n-1}=K_x$.

Firstly, let $t\in T^{n-1}(\BPge{x},\BPg{x};\Delta F(x))$ be an
element that is non-zero only in the 
coordinate indexed by
$\tau=x\tau_{n-1}\cdots\tau_{j+1}\tau_{j-1}\cdots\tau_0$ where
$\cork{\tau_i}=i$. 
We have for $\ss\in N^n_\circ\BP$ that
$dt\cdot \sigma \neq 0$ 
if and only if $\sigma$ is of the form
$x\tau_{n-1}\cdots\tau_{j+1}y\,\tau_{j-1}\cdots\tau_0$,
in which case $dt\cdot \ss=t\cdot\tau$.
Thus, $dt$ is $B_\tau$-constant. As every element of
$T^{n-1}(\BPge{x},\BPg{x};\Delta F(x))$ is a sum of such $t$, we have
$\im d^{n-1}\subset K_x$. 

Conversely, let $s$ be $B_\tau$-constant with value $a$ and with $\tau$
as in the previous paragraph. 
Define $t\in T^{n-1}$ to have value $(-1)^j a$ in the coordinate indexed by
$\tau$ and $0$ elsewhere. Then $dt=s$ and so $K_x\subset
d^{n-1}$. The second isomorphism follows from the first and
(\ref{eq:25}). %\qed
\end{proof}

\begin{remark}\label{remark:nonfree}
The group $A_x$ of Proposition \ref{prop:Kiso} need not be free: let $X$ be
a finite $(n-1)$-dimensional regular CW-complex with homology
$H_{n-2}(X;\Z)$ containing the torsion subgroup $T_{n-2}\not=0$
(for example $X$ is the result of repeatedly suspending $\R P^2$). Let
$\BQ$ be the cell poset of $X$ (see 
\S\ref{section4:2})  and let $\BP$ be $\BQ$
with a unique minimal element $\0$ formally attached. Then,
%considering the interval $\BPg{\0}$:
$$
A_\0
\cong
H^{n-1}(|N^*\BPg{\0}|,\Z)
\cong
H^{n-1}(|N^*\BQ|,\Z)
\cong
H^{n-1}(X,\Z)
\cong
T_{n-2}\oplus\text{free part of }H_{n-1}(X,\Z)
$$
where we have used the fact (see \S\ref{section4:2}) that $X$
and $|N^*\BQ|$ are homeomorphic. 
\end{remark}

Here is our final version of the cellular complex:

\begin{proposition}\label{prop:cellularcomplex.v4}
Let $\BP$ be graded locally finite with a corank function and $F$ a
presheaf on $\BP$. Then
there are isomorphisms
$$
C^n \cong \prod_{\cork{x}=n} 
%\biggl(\bigoplus_{\ss} F(x)\biggr) /K_x
A_x\otimes F(x)
$$
where $A_x$ is 
the abelian group 
having presentation with generators the $\ss\in N^n_\circ \BPge{x}$ and
relations the $\sum_{\sigma \in B_\tau} \sigma = 0$
for each compatible family $B_\tau$ in $N^n_\circ\BPge{x}$. If $\cork{y}=n-1$ and $x\prec y$
then the matrix element $\delta_x^y:A_y\otimes F(y)\rightarrow
A_x\otimes F(x)$ of the differential $\delta:C^{n-1}\rightarrow C^n$
is given by
\begin{equation}
  \label{eq:22}
\delta_x^y:\ss\otimes a\mapsto x\ss\otimes (-1)^n F_x^y(a)  
\end{equation}
where $\ss\in N_\circ^{n-1}\BPge{y}$ is a generator of $A_y$ with $a\in
F(y)$ and $x\sigma$ is the result of pre-appending $x$ to $\ss$. 
\end{proposition}

\begin{figure}
  \centering
\begin{pspicture}(0,0)(12,5.75)
%\showgrid
\rput(-1.5,0.25){
\rput(7.5,2.5){
\rput(0,0){\BoxedEPSF{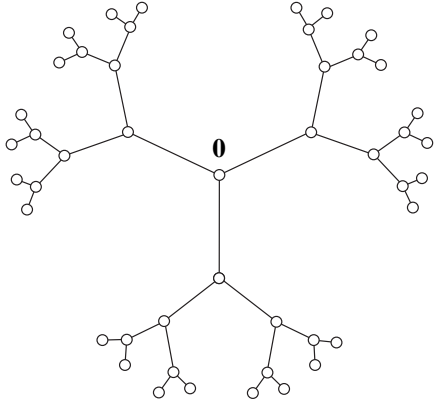 scaled 450}}
}
\rput(7.5,3.2){$\0$}
}
\end{pspicture}
\caption{$3$-valent tree $\BP$ for which $HS^*(\BP;F)$ cannot be
  computed cellularly.}
  \label{fig:tree}
\end{figure}

We finish with an example of a $\BP$ for which $HS^*(\BP;F)\not\cong HC^*(\BP;F)$.

\begin{example}\label{example:tree}
Let $\BP$ be a finite $3$-valent tree with a distinguished vertex $\0$ -- as for example in Figure
\ref{fig:tree} -- 
with vertices ordered by $x\leq y$ when the unique path without
backtracking from $\0$ to $y$ passes through $x$. Then $\BP$ is graded
with the rank of a vertex
the number of edges between it and $\0$.
The maximal elements are called \emph{leaves\/}. 
Assume for simplicity that the leaves are all equidistant from $\0$
(or have the same rank).
By Proposition
\ref{prop:cellularcomplex.v4}, and if the maximum rank $r$ is $>1$, we
have $C^0(\BP;\Delta\Z)$ is
free abelian on the leaves; $C^1(\BP;\Delta\Z)$ is free abelian on the
corank $1$ elements, and $C^i(\BP;\Delta\Z)=0$ for
$i>1$. Thus
$HC^i(\BP;\Delta\Z)=0$ for $i>0$ and
$HC^0(\BP;\Delta\Z)$ is free abelian on the pairs
of leaves. 
On the other hand $\BP$ has a unique minimum $\0$ so that $|N^*\BP|$
is a cone on the space $|N^*\BPg{\0}|$, hence contractible. In
particular $HS^0(\BP;\Delta\Z)\cong\Z$ and $HS^i(\BP;\Delta\Z)=0$ for $i>0$.
\end{example}

%\section{The relation of cellular to (presheaf) cohomology}
\section{Computing cohomology cellularly}
\label{section3}

In general the cohomology groups $\hs^* (\BP;F)$ and $\hc^* (\BP;F)$
are not isomorphic as Example \ref{example:tree} shows.
For a large class of posets however we have $HS^*(\BP;F)\cong
HC^*(\BP;F)$ and so the higher limits can be computed cellularly.
The situation is analogous to
topology: 
if $X$ is a filtered space then one can construct the cellular cochain
complex of $X$, although in general the resulting cellular cohomology is not
isomorphic to the singular cohomology. If a
vanishing condition on the relative (singular) cohomologies of 
successive pairs of the filtration is satisfied then the two are isomorphic.

\begin{definition}
\label{definition:cellular}
Let $\BP$ be graded with a corank function. Then $\BP$ is cellular if
and only if for every presheaf $F$ we have
\begin{equation}
  \label{eq:11}
\hs^i(\BP^n,\BP^{n-1};F)=0\text{ for }i\not=n.  
\end{equation}
\end{definition}

Thus the cohomology of the pair $(\BP^n,\BP^{n-1})$ vanishes in every
degree except the one that carries the cochains of the cellular
complex. 
Using the results of the previous section we have $\BP$ cellular when
\begin{equation}
  \label{eq:12}
\hs^i(\BPge{x},\BPg{x};\Delta F(x))
\cong
\widetilde{\hs}^{i-1}(\BPg{x};\Delta F(x)) =0
\quad
(i\not=\cork{x})   
\end{equation}
for every $x\in\BP$ and every presheaf $F$. Moreover
a locally finite $\BP$ is cellular when 
\begin{equation}
  \label{eq:13}
\hs^i(\BPge{x},\BPg{x};\Delta\Z)
\cong
\widetilde{\hs}^{i-1}(\BPg{x};\Delta\Z) 
\cong
\widetilde{H}^{i-1}(|N^*\BPg{x}|,\Z) =0
\quad
(i\not=\cork{x})   
\end{equation}
for every $x\in\BP$
and with the last term the ordinary reduced cohomology of the space
$|N^*\BPg{x}|$. Cellularity for locally finite $\BP$ thus has nothing
to do with the presheaf. 

\begin{remark}%[non cellular posets]
It is easy to find non-cellular posets, arguing topologically as in
Remark \ref{remark:nonfree} of \S\ref{section2:subsection3}. If $X$ is a regular
CW-complex with non-vanishing cohomology in some non-zero degree $<\dim
X$, with cell poset $\BQ$ (see
\S\ref{section4:2}) and $\BP$ the result of
formally adjoining a 
unique minimal element $\0$ to $\BQ$, then (\ref{eq:13}) fails for
$\BP$ at $x=\0$. 
\end{remark}

We devote \S\ref{section4} to examples of posets that \emph{are\/} cellular.

\begin{remark}
Even if $\BP$ is cellular the cochains of the cellular complex need not
be free: let $X$ be a finite $n$-dimensional regular CW-complex with homology
$H_i(X;\Z)$ the finite group $T\not=0$ in degree $\dim X-1$ and
vanishing in degrees $0<i<\dim
X-1$. Again, suspending $\R P^2$ some number of times provides an
example. If $\BP$ is the result of formally adjoining a unique minimal
element to the cell poset of $X$, then $\BP$ is graded, cellular and
with corank function, and $C^{n+1}\cong T$.
\end{remark}

\begin{remark}
On the other hand if $\BP$ is locally finite cellular and the cochains
of the cellular complex \emph{are\/} free then we have
$$
C^n\cong\prod_{\cork{x}=n} \Z^{\mu_x}\otimes F(x)
$$
where $\mu_x=(-1)^{\cork{x}-1}\mu(x,\1)$ with $\mu$ the M\"{o}bius
function of the poset obtained by adjoining a unique maximum $\1$ to
$\BP$.
This follows by \cite[Proposition 3.8.6]{Stanley12} which interprets
$\mu_x$ in terms of the reduced Euler characteristic of the space
$|N^*\BPg{x}|$, and this characteristic has only one non-zero term by cellularity.
\end{remark}

%{\blue [version of cellular complex with $C^n\cong \prod\Z^\mu\otimes F(x)$]}

\begin{theorem}\label{thm:main}
Let $\BP$ be graded, cellular, locally finite with a corank function and let $F$ be a
presheaf on $\BP$. Then there is an isomorphism
$$
\hs^* (\BP; F)\cong \hc^* (\BP; F).
$$
\end{theorem}

\begin{proof}
\emph{(i).\/} Assume in addition to the conditions stated in the
theorem that
$\BP$ is also finite. We use a spectral sequence by filtering $S^*=
S^*(\BP; F)$ with $F^pS^*=S^*(\BP,\BP^p;F)$. By Lemma \ref{lem:pqr} 
we have a short exact sequence
\begin{equation}
  \label{eq:7}
0\longrightarrow 
S^*(\BP,\BP^{p+1};F)
\longrightarrow
S^*(\BP,\BP^p;F)
\longrightarrow
S^*(\BP^{p+1},\BP^p;F)
\longrightarrow
0
\end{equation}
hence $F^{p+1}S^*$
is a subcomplex of $F^{p}S^*$ and we have a bounded filtration
$$
0\subset F^cS^*\subset\cdots \subset F^{p+1}S^*\subset F^pS^*\subset
\cdots\subset F^{-1}S^*=S^* 
$$
where $c$ is the maximum corank.
The $E_0$ page of the associated spectral sequence has
$$
E_0^{p,q}=\frac{F^{p}S^{p+q}}{F^{p+1}S^{p+q}}
$$
which is just
$S^{p+q}(\BP^{p+1},\BP^{p};F)$ by (\ref{eq:7}). Since the differential on the $E_0$
page is induced by that on $S^*$ we get an $E_1$ page with
$$
E_1^{p,q}=\hs^{p+q}(\BP^{p+1},\BP^{p};F)
$$
When $q=1$ this is simply $C^p$, and by the cellular assumption all
other entries on the $E_1$ page are zero. 
One may check that
on this one line the spectral sequence differential $d^1$ agrees with
the differential in $C^*$. So the spectral sequence collapses at $E_2$ with the
cellular cohomology on the line $q=1$ and hence the result.

\emph{(ii).\/} Returning to the general case of a locally finite
$\BP$, we proceed differently so as to avoid questions about the
convergence of the spectral sequence used in (i). We define a 
projective resolution $B_*\ra\Delta\Z$ in $\prsh(\BP)$ such that
$\hom_{\prsh(\BP)}(B_*,F)=C^*(\BP,F)$. Let $B_n=\bigoplus_x\Upsilon_x$,
the direct sum over the $x$ of corank $n$, where
$\Upsilon_x:=\Upsilon_xA_x$ are Yoneda presheaves (\S\ref{section1:subsection1}) with
$$
A_x=\hom_\Z(\hs^n(\BPge{x},\BPg{x};\Delta\Z),\Z).
$$
$A_x$ is free
and so $\Upsilon_x$, and hence $B_n$, is projective. To
define $\zeta:B_n\rightarrow B_{n-1}$ let $x$ have corank $n$ and $x\prec y$. Define
$\zeta_x^y:S^{*-1}(\BPge{y},\BPg{y};\Delta\Z)\ra
S^*(\BPge{x},\BPg{x};\Delta\Z)$ by
$$
\zeta_x^y s\cdot\ss=
\left\{
\begin{array}{ll}
  (-1)^{n}s\cdot d_k\ss,&\text{if }\ss=xy\ss_{k-2}\ldots\ss_0,\\
  0,&\text{otherwise}
\end{array}
\right.
$$
where $s\in S^{k-1}(\BPge{y},\BPg{y};\Delta\Z)$ and $\ss\in
N^k\BPge{x}$. Then $\zeta_x^y$ is a chain map, inducing a map, which
we will also call $\zeta_x^y$:
$$
A_x=
\hom_\Z(\hs^n(\BPge{x},\BPg{x};\Delta\Z),\Z)
\ra 
\hom_\Z(\hs^{n-1}(\BPge{y},\BPg{y};\Delta\Z),\Z)
=A_y,
$$
and hence a presheaf morphism $\zeta_x^y:\Upsilon_x\ra\Upsilon_y$. Let
$\zeta=\sum_{x\prec y}\zeta_x^y:B_n\ra B_{n-1}$. The sequence
$\cdots\ra B_n\ra B_{n-1}\ra\cdots$ is exact at $B_n$ precisely 
when it is exact pointwise, i.e. for every $x$ the sequence
$\cdots\ra B_n(x)\ra B_{n-1}(x)\ra\cdots$ is exact at $B_n(x)$. But
this sequence is nothing other than the result of applying
$\hom_\Z(-,\Z)$ to the cellular cochain complex
$C^*(\BPge{x};\Delta\Z)$. By local finiteness $\BPge{x}$ is a finite
poset, and is cellular,  and so by part (i) we have
$\hc^*(\BPge{x};\Delta\Z)\cong\hs^*(\BPge{x};\Delta\Z)$, which in turn
is $\cong H^*(\geo{N^*\BPge{x}};\Z)$, and this vanishes outside degree $0$ as
$\geo{N^*\BPge{x}}$ is contractible. 
Thus $H_nB_*(x)$ vanishes when $n>0$ and is $\cong\Z$ when $n=0$. To
augment $B_*$ consider 
$$
B_1\stackrel{\zeta}{\ra}B_0\ra\coker\zeta\ra 0
$$
with the second map the quotient. Then for all $x$ we have
$\coker\zeta(x)\cong H_0B_*(x)\cong\Z$ and for
$x\leq y$ the map
$\coker\zeta(y)\ra\zeta(x)$ can be identified with the identity $\Z\ra\Z$. Thus
$\coker\zeta\cong\Delta\Z$ and we have our augmentation. 

Finally, if $F\in\prsh(\BP)$ then
\begin{align*}
  \label{eq:17}
  \hom_{\prsh(\BP)}(B_n,F)
&=
\hom_{\prsh(\BP)}\biggl(\bigoplus_x\Upsilon_x,F\biggr)
\cong
\prod_x \hom_{\prsh(\BP)}(\Upsilon_x,F)\\
&\cong
\prod_x\hom_\Z(A_x,F(x))
\cong
\prod_x\hom_\Z(A_x,\Z)\otimes F(x)\\
&\cong
\prod_x \hs^n(\BPge{x},\BPg{x};\Delta\Z)\otimes F(x)
\cong
C^n(\BP,F),
\end{align*}
and by (\ref{eq:21}) in \S\ref{section1:subsection1} we have 
$\hom_{\prsh(\BP)}(B_n\stackrel{\zeta}{\ra}B_{n-1},F)\cong
C^{n-1}(\BP;F)\stackrel{\delta}{\ra}C^n(\BP;F)$. 
%\qed 
\end{proof}

The spectral sequence in the proof of Theorem \ref{thm:main}
has its origins in the work of Godement \cite{Godement73}
(see also \cite{Baclawski75}). A special case of the projective
resolution appears in \cite[\S 1.3]{EverittTurner3}.

\section{Examples}
\label{section4}

%We are long overdue some examples. 
In this section we identify some
important classes of posets as coming under the auspices of Theorem
\ref{thm:main} and describe the resulting cellular chain complexes, in
increasing order of complexity.

% \subsection{Posets with a unique maximum}
% \label{section4:1}

% With a non-constant presheaf $F$ things are not so simple. If $\BP$
% has a unique minimum but no unique maximum then we will see in
% \S\ref{section4:3} that $\hs^*(\BP, F)$
% can be very rich indeed. 

% If $\BP$ has a unique maximum $\1$ then things are similar to the
% constant coefficient case, and the proof is made simple by
% computing cellularly. As observed above, 
% $$
% \widetilde{\hs}^{*-1}(\BPg{x};\Delta F(x))\cong\widetilde{H}^{*-1}(\geo{N^*\BPg{x}},F(x))
% $$
% vanishes in every degree for all $x\not=\1$, the unique element
% of corank $0$, and in all non-negative degrees for $x=\1$ with
% $\widetilde{\hs}^{-1}(\BPg{\1};\Delta F(\1))=F(\1)$. The cellular 
% complex is thus $0\ra F(\1)\ra 0$, with the non-zero term in degree $0$. 
% This can be contrasted with the $S^*$ complex which is non-zero in
% every degree.
% Thus:

% \begin{proposition}[c.f. \cite{Mitchell72}*{Corollary 16.2}]\label{prop:maxima}
% Let $\BP$ be graded, cellular, locally finite with corank function and having a unique maximal
% element $\1$. Let $F$ be a presheaf on $\BP$. Then %$\BP$ is
% %cellular with 
% $\hs^i(\BP;F)$ vanishes for $i>0$ and is $\cong F(\1)$
% when $i=0$.
% \end{proposition}

\subsection{Cell posets}
\label{section4:2}

A CW-complex $X$ is \emph{regular\/} if the attaching map of every
cell is a homeomorphism. In this case the \emph{cell poset\/} $\BP_X$ has
elements the cells of $X$ with cells $x\leq y$ iff $\ov{x}\supset \ov{y}$
(note: reverse inclusion). Since the closure of a cell 
meets only finitely many other cells this poset is locally
finite. It is also  graded and if $\dim   
X<\infty$ then the rank function is bounded with corank function given
by $|x|=\dim x$. 
We have a
topology to poset to topology progression given by $X\ra \BP_X\ra |N^*\BP_X|$
where $X$ and $|N^*\BP_X|$ are homeomorphic (see the proof of Theorem
III.1.7 in \cite{Lundell-Weingram69}). 

If $x$ is an $n$-cell then 
$\BPg{x}$ is the cell poset of the induced CW-decomposition of the
boundary $\partial x$, which is itself an $(n-1)$-sphere. Thus
$$
\widetilde{\hs}^{i-1}(\BPg{x};\Delta\Z) 
\cong
\widetilde{H}^{i-1}(\geo{N^*\BPg{x}},\Z) 
$$
vanishes outside degree $i=\cork{x}$ and
$\widetilde{\hs}^{\cork{x}-1}(\BPg{x};\Delta\Z)\cong\Z$.  
Cell posets are thus cellular with
$$
C^n(\BP;F)
%\cong \prod_{\cork{x}=n}F(x)
%\cong \prod_{\cork{x}=n} \biggl(\bigoplus_\ss\Z\biggr)/K_x\otimes F(x)
\cong\prod_{\cork{x}=n}A_x\otimes F(x)
$$
where $A_x\cong\Z$.

Cell posets also enjoy the
$\Diamond$-property: if $u$ is an $(i+1)$-cell and $v$ an $(i-1)$-cell
($0\leq i\leq\dim X$) with
$u<v$, then there are exactly two $i$-cells $z_1,z_2$
with $u<z_i<v$ (if $i=0$ then there is an $u$ but
no $v$, and if $i=\dim X$ then there is a $v$ but no $u$). 

\begin{figure}
  \centering
\begin{pspicture}(0,0)(12,5)
%\showgrid
\rput(-1,0){
\psline[linestyle=dotted,linewidth=1pt](3.5,4.3)(9,4.3)
\psline[linestyle=dotted,linewidth=1pt](3.5,3.6)(5.2,3.6)
\psline[linestyle=dotted,linewidth=1pt](8.6,3)(10.5,3)
\psline[linestyle=dotted,linewidth=1pt](7.9,2.325)(10.5,2.325)
\psline[linestyle=dotted,linewidth=1pt](8.6,1.625)(10.5,1.625)
\rput(4.9,0.25){$x$}\rput(8.5,0.3){$x$}
\rput(3.3,4.3){$0$}\rput(3.3,3.6){$1$}
\rput(10.9,2.3){$i$}\rput(10.9,3){$i+1$}\rput(10.9,1.6){$i-1$}
\rput(7,2.25){
\rput(0,0){\BoxedEPSF{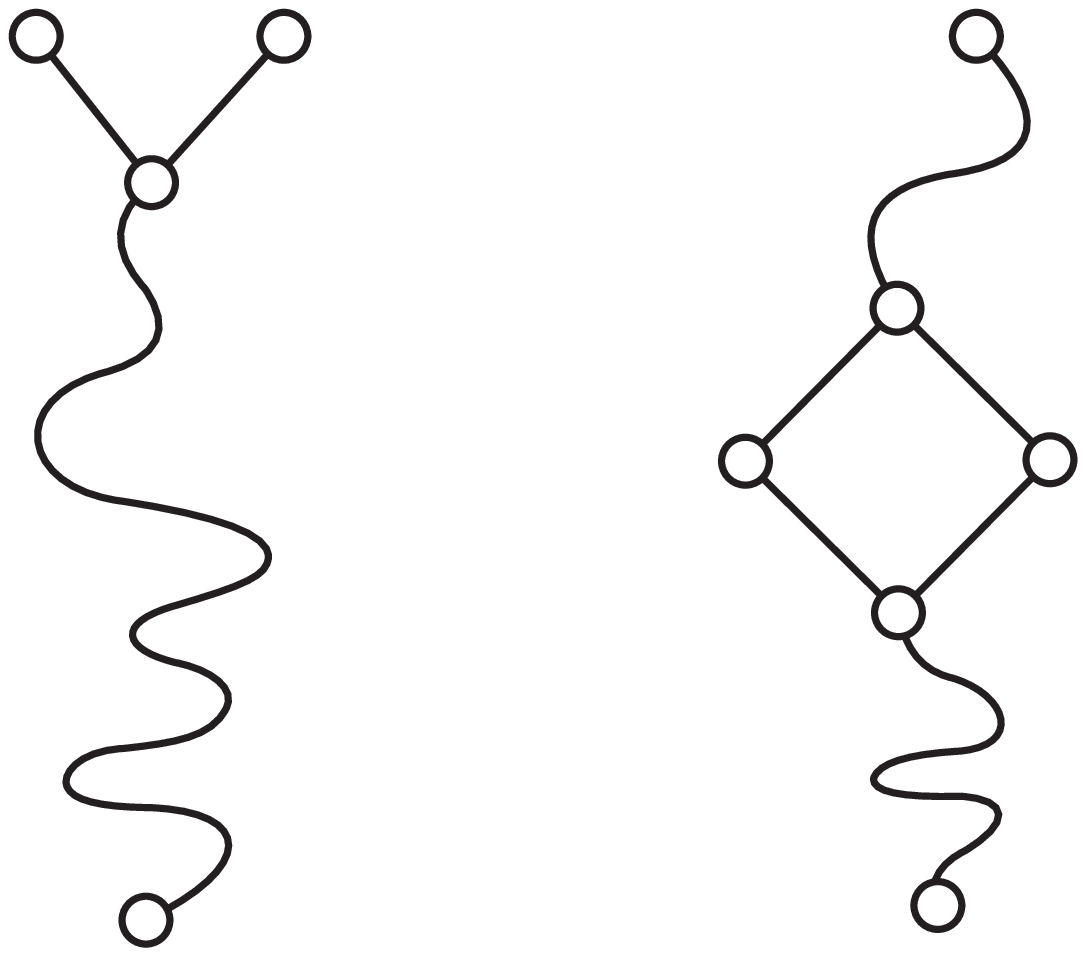 scaled 450}}
}
\rput(5.5,3.6){$u$}\rput(4.7,4.6){$z_1$}\rput(5.8,4.6){$z_2$}
\rput(8.35,1.6){$u$}\rput(8.35,3){$v$}
\rput(7.8,2.05){$z_1$}\rput(9.4,2.05){$z_2$}
}
\end{pspicture}
  \caption{Compatible families in $\BPge{x}$ when $\BP$ is a cell poset.}
  \label{fig:compatible.families}
\end{figure}

A compatible family $B_\tau$ in $\BPge{x}$ thus has one of the two forms illustrated in
Figure \ref{fig:compatible.families}. The group
$A_x$ of Proposition \ref{prop:cellularcomplex.v4} thus has presentation with generators the
$\ss\in N^n_\circ\BPge{x}$ and
relations of the form $\ss+\ss'=0$, where $\ss,\ss'$ are the two poset
sequences running around each side of a $\Diamond$.

We have the description (\ref{eq:22}) of the matrix element
$\delta_x^y:A_y\otimes F(y)\rightarrow A_x\otimes F(x)$ of the
differential, but we can also explicitly describe it as a map
$\Z\otimes F(y)\rightarrow \Z\otimes F(x)$ as follows.
% {\blue Any $\ss\in N^n_\circ\BPge{x}$ freely generates the group
%   $A_x\cong\Z$.} In particular for any other $\ss'\in N^n_\circ\BPge{x}$ we
% have $\ss'=\pm\ss$. 
If $\ss,\ss'\in N^n_\circ\BPge{x}$ then $\ss$ can be turned into
$\ss'$ by successively moving poset sequences across $\Diamond$'s -- i.e.:
replacing one sequence in Figure \ref{fig:compatible.families} by the
other (sketch of proof: $\BPge{x}$ is the cell poset of an
$n$-ball with the induced decomposition of the bounding
$(n-1)$-sphere; the $\ss\in N^n_\circ\BPge{x}$ correspond to simplices
in the barycentric subdivision and moving them across $\Diamond$'s
corresponds to exchanging simplices sharing a common face of dimension
$n-2$). In particular $\ss'=\pm\ss$, with the sign 
determined by the parity number of such maneuvers, and $A_x\cong\Z$ is
freely generated by any of the $\ss\in N^n_\circ\BPge{x}$.

For each $x$, fix a free generator $\ss_x$ of
$A_x$ and for each $x\prec y$ let $[x,y]=\pm 1$ be determined by
\begin{equation}
  \label{eq:27}
x\ss_y=[x,y]\ss_x,  
\end{equation}
where $x\ss_y$ is the result of pre-appending $x$ onto $\ss_y$. 
If $x,y,y',z$ form a $\Diamond$-configuration then one can check that
\begin{equation}
  \label{eq:9}
[x,y][y,z]=-[x,y'][y',z].  
\end{equation}
The
matrix element $\delta_x^y:A_y\otimes F(y)\ra A_x\otimes F(x)$ is then
given by
$$
\delta_x^y:\ss_y\otimes a
\mapsto
x\ss_y\otimes (-1)^n F_x^y(a)
=[x,y]\ss_x\otimes (-1)^n F_x^y(a).
$$
By \cite[Chapter IX, Theorem 7.2]{Massey91} orientations can be chosen
for the cells of $X$ in such a way that the $[x,y]$ -- which are
defined above in a
purely combinatorial way -- are the incidence numbers of the cells. 

\subsection{Posets with unique extrema and Khovanov homology}
\label{section4:3}

If $\BP$ is a poset with a unique extremal -- that is, maximal
  or minimal -- 
element then the classifying space $\geo{N^*\BP}$ is contractible;
indeed if $x$ is the extremal element then $\geo{N^*\BP}$
is a cone on $|N^*(\BP\setminus x)|$. If we have a constant presheaf $F=\Delta A$ on $\BP$ then
$\hs^i(\BP;\Delta A)\cong H^i(\geo{N^*\BP},A)$ 
vanishes for $i>0$ and is $\cong A$ 
in degree $0$. 

More generally if $\BP$ has a unique maximum $x$ then the limit functor
$$F\mapsto \invlim_\BP F$$ 
is naturally isomorphic to the evaluation
functor $F\mapsto F(x)$, which is exact. Hence the higher limits $\hs^i(\BP;F)$
vanish for $i>0$ here also.

% We showed in \S\ref{section4:1} that posets with a unique maximum have cohomology
% concentrated in degree zero. With constant
% coefficients the presence of a unique mininum gives a similar
% result.
If instead $\BP$ has a unique minimum, but no unique maximum, then
  given an interesting enough presheaf 
the higher limits can 
be very rich. A good source of examples comes from the Khovanov
homology \cite{Khovanov00} mentioned in the Introduction: we describe
in \cite[Theorem 1]{EverittTurner3} how 
the Khovanov homology of a link diagram with $n$ crossings arises as the cohomology
of the cell poset of the
suspension $X$ of an $(n-1)$-simplex equipped with the Khovanov
presheaf $F_{\kern-0.5mmKh}$ -- see Figure \ref{fig:orange} (and 
also \cite{EverittTurner1}). The 
cellular cochain complex $C^*(\BP_X;F_{\kern-0.5mmKh})$ is then the standard cube
complex found in Khovanov homology, and the $[x,y]$ of
(\ref{eq:27}) are the signs ``sprinkled'' on the cube to make its faces anti-commute.

\begin{figure}
  \centering
\begin{pspicture}(0,0)(12,4.5)
\rput(-1,0){
\rput(4.6,1.8){
\rput(0,0){\BoxedEPSF{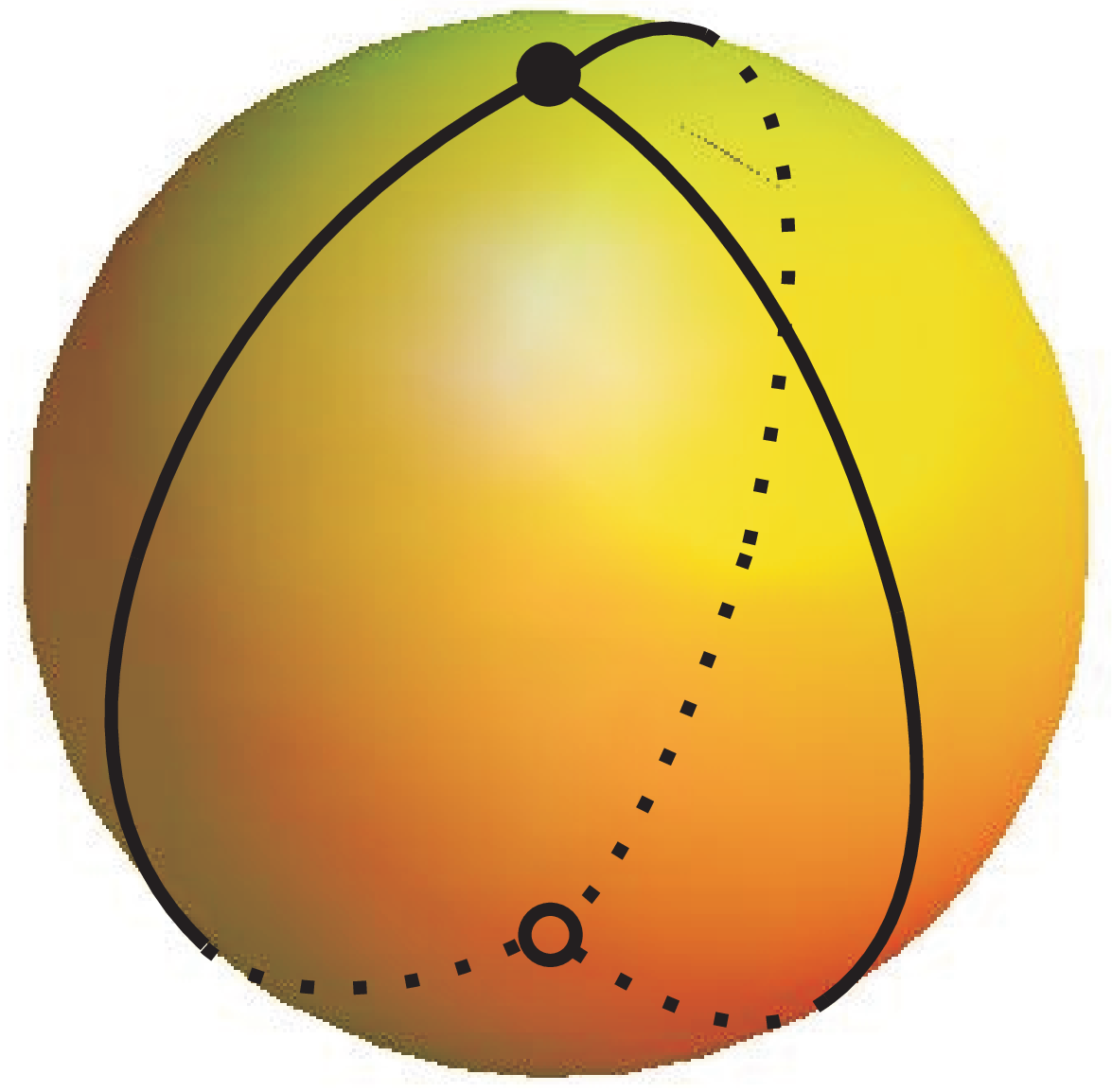 scaled 300}}
\rput(0.1,1.4){$\1$}
\rput(0.15,-0.7){$\1'$}
}
\rput(9.2,2){
\rput(0,0){\BoxedEPSF{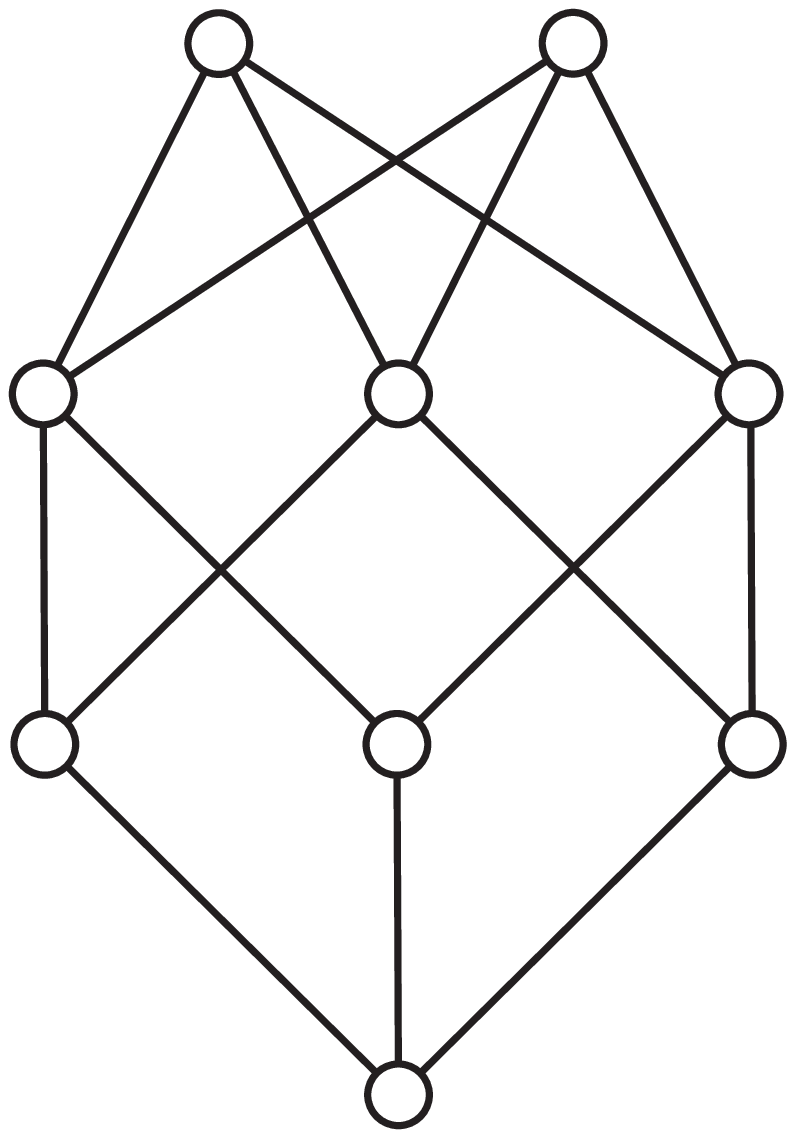 scaled 300}}
\rput(-0.55,1.9){$\1$}
\rput(0.55,1.9){$\1'$}
%\rput(0,-1.9){$\0$}
}}
%\showgrid
\end{pspicture}
\caption{The $X$ \emph{(left)} and $\BP_X$ \emph{(right)} for the Khovanov homology of a link
  diagram with $3$ crossings.}
  \label{fig:orange}
\end{figure}

\subsection{The Bruhat order and the symmetric group}
\label{section4:4}

This is another example of a cell poset, arising from a partial order
on a finite Coxeter group. We illustrate with a particular example.

Let $S_{\kern-0.5mm n}$ be the symmetric group and write an $x\in S_{\kern-0.5mm n}$ as a
string $x=x(1)\cdots x(n)$. %or as a product of disjoint cycles as convenient. 
Then $S_{\kern-0.5mm n}\setminus\id$ can be given the structure 
of a cell poset in the following way. If $x,y\in S_{\kern-0.5mm n}\setminus\id$  % and
% $t=(i,j)$ is a transposition with $i<j$,
then write $x\ra y$
if
$$
x=x(1)\cdots i\cdots j\cdots x(n)
\text{ where }i>j\text{ and }
y=x(1)\cdots j\cdots i\cdots x(n).
$$ 
%$x(i)>x(j)$ and $y=x(1)\cdots x(j)\cdots x(i)\cdots x(n)$.
Define $x\leq y$ when there are $x_i$ with
$$
x=x_0
\ra
x_1
\ra
\cdots
\ra
x_k=y.
$$
The resulting $\leq$ is called the \emph{Bruhat order\/} on $S_{\kern-0.5mm n}$
(actually, our Bruhat order is the 
opposite of that normally found in the literature, but the Bruhat
order is isomorphic to its opposite anyway). For basic facts
concerning the Bruhat order, including some of the constructions
below, see \cite[Chapter 2]{Bjorner-Brenti05}. The corank function is
$\cork{x}=\ell(x)-1$, where $\ell(x)$ is the number
of inversions in $x$: pairs $i>j$ with $x=x(1)\cdots i\cdots j\cdots x(n)$. 
The poset $\BP=S_{\kern-0.5mm n}\setminus\id$
has maxima the $n-1$ transpositions $s_i=1\cdots i+1,i\cdots n$
with $\ell(s_i)=1$ and unique minimum
the permutation $x_0=n\cdots 21$ with $\ell(x_0)=\binom{n}{2}$. 
There is then a regular CW decomposition $X$ of the $(\ell(x_0)-1)$-ball with
$\BP_X=\BP$. %\cite{Bjorner-Brenti05}*{Theorem 2.7.12}.

To describe the cellular complex $C^*(\BP;F)$ for a presheaf $F$ on $\BP$
we need only give a free
generator $\ss_x$ for the group $A_x$ as in
\S\ref{section4:2} and determine the signs
$[x,y]$ of (\ref{eq:27})
for all $x\prec y$ in the Bruhat order.
%For $x\in S_{\kern-0.5mm n}$, the $y$ covering $x$ arise in the
%following way: 
Let $x=x(1)\cdots i\cdots j\cdots x(n)$ where $i>j$ and call $(i,j)$ a \emph{swap pair\/}
if for each $k$ of the string appearing between $i$ and $j$ we have either
$k<j$ or $k>i$. Then if $y=x(1)\cdots
j\cdots i\cdots x(n)$ we have $x\prec y$, and all the $y$ covering $x$
arise by interchanging swap pairs in this way. Totally order
pairs by
$(n,n-1)>\cdots>(n,2)>\cdots>(3,2)>(n,1)>\cdots>(3,1)>(2,1)$, and 
restrict this ordering to the swap pairs. 
Let
$\ss_x=x\prec\ss_{n-1}\prec\cdots\prec\ss_0$ where $\ss_{i-1}$ is the
result of interchanging the minimal swap pair of $\ss_i$. For example
if $x=4321\in S_{\kern-0.5mm 4}$ then
$$
\ss_x=43\underline{21}\prec 4\underline{31}2\prec
\underline{41}32\prec 14\underline{32}\prec 1\underline{42}3\prec
1243  
$$
with the minimal swap pairs underlined (and $12\underline{43}\prec
1234=\id$). Now to the signs. % $[x,y]$ are determined inductively
If $x\prec
y$ with $y$ the result of interchanging the minimal swap pair in $x$,
then clearly $[x,y]=1$. If now $x$ has corank $1$ and $y$ is the
result of interchanging a non-minimal swap pair in $x$ then
$[x,y]=-1$ (as $x\ss_y+\ss_x=0$ via a relation of the form given on
the left of Figure \ref{fig:compatible.families}). For a general
covering $x\prec y$
it is possible to find a $\Diamond$-configuration $x,y,y',z$,
so that $[x,y]=-[x,y'][y,z][y',z]$ by (\ref{eq:9}), where
$[x,y'],[y,z]$ and $[y',z]$ are already 
known, the last two by induction on the corank. We leave the details
to the reader. Figure \ref{fig:Bruhat:Eg} illustrates the case $n=4$.

\begin{figure}
  \centering
\begin{pspicture}(0,0)(14,6.5)
%\showgrid
\rput(-1,0){
\rput(7,3.25){
\rput(0,0){\BoxedEPSF{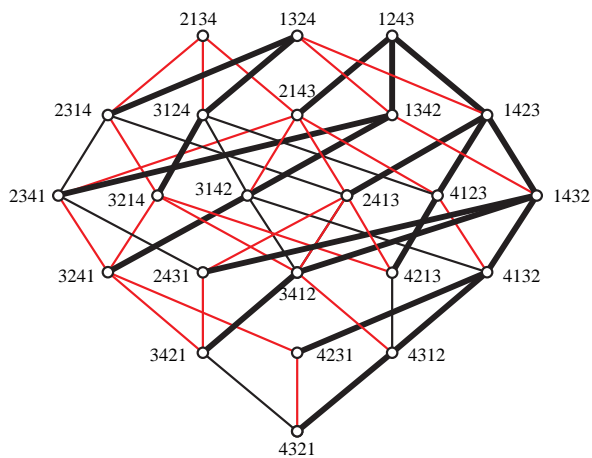 scaled 550}}
\rput(0,-2.85){${\scriptstyle 4321}$}
\rput(-1.70,-1.6){${\scriptstyle 3421}$}\rput(0.5,-1.6){${\scriptstyle
    4231}$}\rput(1.7,-1.6){${\scriptstyle 4312}$} 
\rput(-2.9,-0.55){${\scriptstyle
    3241}$}\rput(-1.65,-0.55){${\scriptstyle
      2431}$}\rput(0,-0.8){${\scriptstyle 3412}$} 
\rput(1.65,-0.55){${\scriptstyle 4213}$}\rput(2.95,-0.55){${\scriptstyle 4132}$}
\rput(-3.55,0.5){${\scriptstyle 2341}$}\rput(-2.25,0.45){${\scriptstyle
      3214}$}\rput(-1.1,0.6){${\scriptstyle 3142}$} 
\rput(1.1,0.45){${\scriptstyle 2413}$}\rput(2.25,0.55){${\scriptstyle
      4123}$}\rput(3.6,0.5){${\scriptstyle 1432}$} 
\rput(-2.95,1.65){${\scriptstyle
    2314}$}\rput(-1.65,1.65){${\scriptstyle
      3124}$}\rput(0,1.9){${\scriptstyle 2143}$} 
\rput(1.65,1.65){${\scriptstyle 1342}$}\rput(2.95,1.65){${\scriptstyle 1423}$}
\rput(-1.3,2.85){${\scriptstyle 2134}$}\rput(0,2.85){${\scriptstyle
      1324}$}\rput(1.3,2.85){${\scriptstyle 1243}$} 
}}
\end{pspicture}
  \caption{$S_{\kern-0.5mm 4}\setminus\id$ equipped with the Bruhat
    order: the thick edges give the generators $\ss_x$; the black
    edges (both thick and thin) are the $x\prec y$ with $[x,y]=1$ and the red edges 
    are the $x\prec y$ with $[x,y]=-1$.} 
  \label{fig:Bruhat:Eg}
\end{figure}

\subsection{Geometric lattices}
\label{section4:5}

A \emph{lattice\/} is a poset $\BP$ such that any two elements
$x$ and $y$ have a supremum (or join) $x\vee y$ and an infimum (or
meet) $x\wedge y$. $\BP$ has \emph{finite length\/} is there is an
absolute bound on the number of elements in any poset chain
$x_0\leq\cdots\leq x_n$. If $\BP$ has finite length and a unique
minimum $\0$, then define a grading by taking $\rk(x)$ to be the
supremum of the lengths of all poset chains from $\0$ to $x$. 
$\BP$ is a \emph{geometric\/} lattice if every element can be
expressed as a join of elements of rank $1$ (called \emph{atoms\/}) and
for any $x,y$ we have
\begin{equation}
  \label{eq:19}
\rk(x\vee y)+\rk(x\wedge y)\leq \rk(x)+\rk(y)
\end{equation}
% (or equivalently $\cork{x\vee y}+\cork{x\wedge
%   y}\geq\cork{x}+\cork{y}$). 
The motivating example is the linear subspaces of a vector space $V$
over some field $k$, ordered by reverse inclusion. % (so that $\cork{x}=\dim
% x$).
See \cite[Chapter IV]{Birkhoff79} or \cite[Chapter 3]{Stanley12} for
general facts about geometric lattices. 

Let $\BP$ be a locally finite geometric lattice -- in particular
$\BP$ is finite and hence also has a unique maximum $\1$, the join of
the elements of $\BP$. In the light of Section \ref{section4:3}, let
$\BQ=\BP\setminus\1$. For every $x\in\BP$ the interval $\BPge{x}$ is
also a geometric lattice. Let $\mu_x:=(-1)^{\cork{x}-1}\mu(x,\1)$ where $\mu$
is the M\"{o}bius function of $\BP$ and $\cork{\cdot}$ is the corank
function of $\BQ$. For any $x$ the space $\geo{N^*\BQ_{>x}}$ has
the homotopy type of a bouquet of $\mu_x$ spheres of dimension
$\cork{x}-1$
(\cite{Quillen78}, see also
\cite[Theorem 4.109]{Orlik-Terao92} and \cite{Bjorner82}),
hence 
$$
\widetilde{\hs}^{\cork{x}-1}(\BQ_{>x};\Delta\Z) \cong \Z^{\mu_x}
%\quad
%(i=\cork{x})   
$$
and the homology vanishes in all other degrees. Thus geometric lattices (minus
their maximal elements) are cellular, and for any presheaf $F$ on $\BQ$
we have %by version $4$ of the cellular chain complex that
$$
C^n(\BQ;F)
\cong\bigoplus_{\cork{x}=n}A_x\otimes F(x)
$$
where $A_x\cong\Z^{\mu_x}$. One can find explicit free generators for $A_x$ using
$R$-labelings \cite[Theorem 3.13.2]{Stanley12} and hence an explicit
description of the differential from Proposition \ref{prop:cellularcomplex.v4}.

%% References
%%
%% Following citation commands can be used in the body text:
%% Usage of \cite is as follows:
%%   \cite{key}          ==>>  [#]
%%   \cite[chap. 2]{key} ==>>  [#, chap. 2]
%%   \citet{key}         ==>>  Author [#]

%% References with bibTeX database:

%\bibliographystyle{model1-num-names}
%\bibliography{<your-bib-database>}

\bibliography{Everitt_Turner}{}
\bibliographystyle{plain}

%% Authors are advised to submit their bibtex database files. They are
%% requested to list a bibtex style file in the manuscript if they do
%% not want to use model1-num-names.bst.

%% References without bibTeX database:

%\begin{thebibliography}{0}
%
%% \bibitem must have the following form:
%%   \bibitem{key}...
%%
%
% \bibitem{Abramenko_Brown08}
% {\bibname Peter Abramenko and Kenneth S. Brown},
% {\em Buildings} (Springer-Verlag, New York, 2008).
% %
% \bibitem{Borel91}
%  {\bibname Armand Borel},
%  {\em Linear algebraic groups} (Springer-Verlag, New York, 1991).
% %
% \bibitem{Bourbaki02}
%  {\bibname Nicolas Bourbaki},
%  {\em  Lie groups and Lie algebras. Chapters 4--6} (Springer-Verlag,
%  Berlin, 2002).
%
%
%\end{thebibliography}

\end{document}